\documentclass[11pt]{article}
\usepackage[top=1in, bottom=1in, left=1in, right=1in]{geometry}

\usepackage[linktocpage,colorlinks,linkcolor=blue,anchorcolor=blue,citecolor=blue,urlcolor=blue,pagebackref]{hyperref}
\usepackage{euscript,mdframed}
\usepackage{microtype,todonotes,relsize}
\usepackage{amsmath,amsthm}
\usepackage{algpseudocode}

\usepackage{accents}
\usepackage{comment,url,graphicx,relsize}
\usepackage{amssymb,amsfonts,amsmath,amsthm,amscd,dsfont,mathrsfs,mathtools,nicefrac, bm}
\usepackage{float,psfrag,epsfig,color,xcolor,url}
\usepackage{epstopdf,bbm,mathtools,enumitem}
\usepackage{subfigure}
\usepackage[ruled,vlined]{algorithm2e}
\usepackage{tablefootnote}
\graphicspath{{Plots/}}
\usepackage{diagbox}
\usepackage{tikz}
\usetikzlibrary{calc,patterns,angles,quotes}
\usepackage{makecell}
\usepackage{multirow}
\usepackage{booktabs}
\usepackage{subcaption}

\def\reals{\mathbb{R}} 

\newcommand{\inner}[2]{\langle{#1},{#2}\rangle}

\newcommand{\grad}{\operatorname{grad}} 

\newcommand{\B}{\mathcal{B}}
\renewcommand{\P}{\mathcal{P}}
\newcommand{\M}{\mathcal{M}}
\newcommand{\N}{\mathcal{N}}
\newcommand{\E}{\mathbb{E}}

\providecommand{\argmin}{\mathop\mathrm{arg min}}

\providecommand{\tr}{\mathop\mathrm{Tr}}

\newcommand{\retr}{\operatorname{Retr}}

\newcommand{\T}{\mathcal{T}}
\renewcommand{\d}{\operatorname{d}}

\ifdefined\nonewproofenvironments\else
\ifdefined\ispres\else
\renewenvironment{proof}{\noindent\textbf{Proof.}\hspace*{.3em}}{\qed\\}
\newenvironment{proof-sketch}{\noindent\textbf{Proof Sketch}
  \hspace*{0.em}}{\qed\bigskip\\}
\newenvironment{proof-idea}{\noindent\textbf{Proof Idea}
  \hspace*{0.em}}{\qed\bigskip\\}
\newenvironment{proof-of-lemma}[1][{}]{\noindent\textbf{Proof of Lemma {#1}.}
  \hspace*{0.em}}{\qed\\}
\newenvironment{proof-of-corollary}[1][{}]{\noindent\textbf{Proof of Corollary {#1}.}
  \hspace*{0.em}}{\qed\\}
\newenvironment{proof-of-theorem}[1][{}]{\noindent\textbf{Proof of Theorem {#1}.}
  \hspace*{0.em}}{\qed\\}
\newenvironment{proof-attempt}{\noindent\textbf{Proof Attempt}
  \hspace*{0.em}}{\qed\bigskip\\}

\fi
\newtheorem{theorem}{Theorem}[section]
\newtheorem{lemma}{Lemma}[section]
\newtheorem{corollary}{Corollary}[section]
\newtheorem{proposition}{Proposition}[section]
\newtheorem{assumption}{Assumption}[section]
\newtheorem{remark}{Remark}[section]

\newtheorem{definition}{Definition}[section]

\usetikzlibrary{calc}

\allowdisplaybreaks
\usepackage[authoryear,round]{natbib}
\renewcommand*{\backref}[1]{\ifx#1\relax \else Page #1 \fi}
\renewcommand*{\backrefalt}[4]{%
  \ifcase #1 \footnotesize{(Not cited.)}%
  \or        \footnotesize{(Cited on page~#2.)}%
  \else      \footnotesize{(Cited on pages~#2.)}%
  \fi
}

\newcommand*{\colorboxed}{}
\def\colorboxed#1#{%
  \colorboxedAux{#1}%
}
\newcommand*{\colorboxedAux}[3]{%
  \begingroup
    \colorlet{cb@saved}{.}%
    \color#1{#2}%
    \boxed{%
      \color{cb@saved}%
      #3%
    }%
  \endgroup
}
\numberwithin{equation}{section}
\usepackage{xcolor} 
\usepackage{marginnote}

\newcommand{\todol}[2][]{{%
 \let\marginpar\marginnote
 \reversemarginpar
 \renewcommand{\baselinestretch}{0.8}%
 \todo[color=yellow]{#2}}}

\title{
On Relatively Smooth Optimization over Riemannian Manifolds
}
\author{
{Chang He} \thanks{School of Information Management and Engineering, Shanghai University of Finance and Economics; Department of Industrial and System Engineering, University of Minnesota. \texttt{ischanghe@gmail.com}}
\and
{Jiaxiang Li} \thanks{Department of Electrical and Computer Engineering, University of Minnesota.  \texttt{li003755@umn.edu}}
\and
{Bo Jiang} \thanks{School of Information Management and Engineering, Shanghai University of Finance and Economics. \texttt{isyebojiang@gmail.com}}
\and
{Shiqian Ma} \thanks{Department of Computational Applied Mathematics and Operations Research, Rice University. \texttt{shiqian.ma@rice.edu}}
\and
{Shuzhong Zhang} \thanks{Department of Industrial and System Engineering, University of Minnesota. \texttt{zhangs@umn.edu}}
}
\date{}

\begin{document}
\maketitle

\begin{abstract}
We study optimization over Riemannian embedded submanifolds, where the objective function is relatively smooth in the ambient Euclidean space. Such problems have broad applications but are still largely unexplored. We introduce two Riemannian first-order methods, namely the retraction-based and projection-based Riemannian Bregman gradient methods, by incorporating the Bregman distance into the update steps. The retraction-based method can handle nonsmooth optimization; at each iteration, the update direction is generated by solving a convex optimization subproblem constrained to the tangent space. We show that when the reference function is of the quartic form $h(x) = \frac{1}{4}\|x\|^4 + \frac{1}{2}\|x\|^2$, the constraint subproblem admits a closed-form solution. The projection-based approach can be applied to smooth Riemannian optimization, which solves an unconstrained subproblem in the ambient Euclidean space. Both methods are shown to achieve an iteration complexity of $\mathcal{O}(1/\epsilon^2)$ for finding an $\epsilon$-approximate Riemannian stationary point. When the manifold is compact, we further develop stochastic variants and establish a sample complexity of $\mathcal{O}(1/\epsilon^4)$. Numerical experiments on the nonlinear eigenvalue problem and low-rank quadratic sensing problem demonstrate the advantages of the proposed methods.
\end{abstract} 

\section{Introduction}
In this work, we study the following constrained composite optimization problem:
\begin{equation}\label{eq.main}
\begin{aligned}
    \min_{x \in \reals^n} \quad &F(x) = f(x) + g(x) \\
    \operatorname{s.t.} \quad  &x \in \M \subseteq \reals^n,
\end{aligned}
\end{equation}
where $\M$ is a Riemannian embedded submanifold of $\reals^n$, $f: \reals^n \to \reals$ is continuously differentiable and may be nonconvex, and $g: \reals^n \to \reals$ is a convex, continuous (possibly nonsmooth) function. Here, convexity and
smoothness are interpreted as the function is being considered in the ambient Euclidean space. Moreover, the objective function $F(\cdot)$ is bounded below on $\M$, i.e., $F^* = \inf_{x \in \M} F(x) > -\infty$. Such a constrained optimization problem has attracted considerable attention due to its numerous applications, including principal component analysis, low-rank matrix completion, and dictionary learning \citep{absil2009optimization,vandereycken2013low,boumal2015low,sun2015complete,cherian2016riemannian,liu2019quadratic,boumal2023introduction,mishra2019riemannian,li2024riemannian}. By exploiting the geometric structure of the manifold, such as through a suitable retraction, manifold optimization problems can be tackled as unconstrained problems, often resulting in stronger convergence guarantees.


To measure the efficiency of Riemannian optimization methods, one typically considers the iteration complexity, which refers to the number of iterations needed to obtain an approximate solution. When optimizing over a Riemannian embedded submanifold $\M$, most complexity analyses rely on a Riemannian version of the descent property (e.g., Property A3 in \citet{boumal2019global}), closely analogous to the Euclidean case. Consequently, the same iteration complexity bounds as in their unconstrained Euclidean counterparts typically hold. This descent property is usually established by combining Euclidean gradient Lipschitz continuity with retraction inequalities. However, standard Lipschitz gradient continuity can be sometimes restrictive. Even the simple polynomial \(f(x)=x^{4}\) lacks a Lipschitz continuous gradient, and the widely used log-barrier function in interior-point methods fails to satisfy this condition either \citep{nesterov1994interior,hinder2024worst,jiang2024barrier}. Although it is often possible to argue that the iterates remain within a compact set, the resulting Lipschitz constant can become too large that the corresponding complexity bound offers limited practical insight. This motivates us to develop a more general framework for Riemannian optimization, extending beyond the standard Lipschitz gradient assumption.  

In this paper, we adopt the notion of \emph{relative smoothness}\footnote{Some authors, e.g., \cite{takahashi2024approximate}, refer to this property as the ``adaptable smoothness property".}, which is defined relative to a reference function \citep{bauschke2017descent,lu2018relatively}. The formal definition is as follows:
\begin{definition}[Relative smoothness]
   Given a differentiable convex function $h$, referred to as the reference function, function $f$ is said to be $L$-smooth relative to $h$ if, for all $x, y \in \reals^n$,
   \begin{align*}
      f(y) \ \le \ f(x) + \inner{\nabla f(x)}{y-x} + L \cdot D_h(y,x),  
   \end{align*}
where $D_h(y,x)$ is the Bregman distance induced by $h$, defined as
\begin{align*}
   D_h(y,x) \ \triangleq \ h(y) - h(x) - \inner{\nabla h(x)}{y - x}. 
\end{align*}
\end{definition}

Clearly, choosing $h(x) = \frac{1}{2}\|x\|^2$ recovers the standard notion of gradient Lipschitz continuity. In certain applications, carefully selecting the reference function $h$ can yield a more accurate local approximation \citep{bolte2018first}. Below, we briefly highlight several representative applications for which the objective function is relatively smooth and the feasible region is a Riemannian manifold.

\subsection{Motivating examples}
\paragraph{Polynomial optimization over the Stiefel manifold.} Constrained polynomial optimization is a widely studied class of problems, capturing applications in signal processing, machine learning, and control \citep{li2012approximation}. A prototypical instance over the Stiefel manifold is the nonlinear eigenvalue problem, which arises in electronic-structure calculations, such as Kohn–Sham and Hartree–Fock energy‐minimization models \citep{cai2018eigenvector,yang2009convergence,liu2014convergence}. In particular, discretising a one-dimensional Kohn–Sham equation leads to
\begin{equation}\label{eq.ne}
\begin{aligned}
  \min_{X\in\mathbb{R}^{m\times p}} \quad 
      &f(X) = \frac{1}{2}\tr(X^\top L X) + \frac{\beta}{4} \rho_X^\top L^{\dagger} \rho_X \\
    \operatorname{s.t.}\quad & X^\top X = I_p,
\end{aligned}
\end{equation}
where \(\rho_X \triangleq \operatorname{diag}(XX^{\top})\) collects the orbital densities,
\(L\in\mathbb{R}^{m\times m}\) is the tridiagonal matrix with \(2\) on the main diagonal and \(-1\) on the first sub- and super-diagonals,
\(L^{\dagger}\) denotes its pseudo-inverse, and \(\beta>0\) is a parameter. Because the Hessian of \(f\) grows as a polynomial in \(\|X\|\), we can invoke the result in \citet{lu2018relatively}: if  
\(\|\nabla^2 f(X)\|\le p_r(\|X\|)\) for a univariate polynomial \(p_r\) of degree \(r\),  
then \(f\) is relatively smooth with respect to  $h(X) = \frac{1}{r+2}\|X\|^{r+2} + \frac{1}{2}\|X\|^2$. Specifically, when $p = 1$, the Stiefel manifold reduces to the sphere, yielding the classical polynomial optimization with sphere constraints.

\paragraph{Low-rank quadratic sensing problem.} The quadratic sensing problem is a fundamental optimization problem arising in statistical models (Section 4 in \cite{chi2018nonconvex}). It appears in various applications, such as covariance sketching for streaming data; see, e.g.,~\cite{chen2015exact,cai2015rop}. In this problem, we have access to $N$ measurements of a rank-$r$ matrix $\Sigma = X_* X_*^\top$ with $X_* \in \mathbb{R}^{m\times r}$:
\begin{align*}
   c_j = \|X_*^\top y_j\|^2 = y_j^\top \Sigma y_j, j = 1, \ldots, N,
\end{align*}
where $y_j \in \reals^m$ are known design vectors. The goal is to reconstruct $X_*$ from these quadratic measures. Mathematically, this task can be formulated as the following optimization problem:
\begin{equation}\label{eq.low-rank recover}
    \begin{aligned}
       \min_{X\in\mathbb{R}^{m\times r}} \quad 
      &f(X)=\frac{1}{2} \sum_{j=1}^N\left(\|X^\top y_j\|^2-c_j\right)^2 \\
    \operatorname{s.t.}\quad & \operatorname{rank}(X) = r. 
    \end{aligned}
\end{equation}
Note that $f$ is a fourth-degree polynomial in the entries of $X$. Consequently, the objective function is relatively smooth with respect to the reference function $h(X)=\tfrac14\|X\|^{4}+\tfrac12\|X\|^{2}$.

\paragraph{Low‑rank Poisson matrix completion.} Low-rank Poisson matrix completion \citep{cao2015poisson,mcrae2021low} seeks to recover a rank-\(r\) matrix \(X \in \mathbb{R}^{m\times p}\) from partial nonnegative integer observations \(\{Y_{ij}\}_{(i,j)\in\Omega}\). The observations are Poisson counts of the observed matrix entries:
\[
  Y_{ij} \sim\mathrm{Poisson}\left(X_{ij}\right),\quad (i,j)\in\Omega,
\]  
where \(\Omega\subseteq\{1,\dots,m\}\times\{1,\dots,p\}\) is the set of observed entries. We recover the matrix \(X\) via the maximum likelihood formulation; specifically, we minimize the negative log-likelihood:
\begin{align*}
    \min_{X\in\mathbb{R}^{m\times p}} \quad 
      &f(X) = \sum_{(i,j)\in\Omega}\left(X_{ij} - Y_{ij}\log\left(X_{ij}\right)\right) \\
    \operatorname{s.t.}\quad &\operatorname{rank}(X)=r, \ X_{ij}>0, \ \forall(i,j)\in\Omega,
\end{align*}
which is naturally posed on the embedded submanifold defined as  
\(\mathcal{M}=\{X\in\mathbb{R}^{m\times p}:\operatorname{rank}(X)=r,X_{ij}>0, \ \forall(i,j)\in\Omega\}\). Notice that the Euclidean Hessian \(\nabla^2 f(X)\) has entries \(Y_{ij}/X^2_{ij}\) for \((i,j)\in\Omega\), which become unbounded as \(X_{ij}\to0^+\), and thus \(f\) does not admit a global Lipschitz constant for its gradient. To identify the relative smoothness, we instead choose the logarithmic barrier function $h(X) = -\sum_{(i,j)\in\Omega} \log\left(X_{ij}\right)$ as the reference function. Then, \(f\) is relatively smooth with respect to \(h\).

\subsection{Related works}

\paragraph{Optimization over Riemannian manifolds.} 
For smooth Riemannian optimization problems, i.e., \(g \equiv 0\) in \eqref{eq.main}, first-order methods are popular choices, such as Riemannian gradient descent and its variants, e.g., Riemannian conjugate gradient methods \citep{sato2022riemannian}. When Hessian information is accessible, second-order methods such as Riemannian trust-region method \citep{absil2007trust} and cubic regularized Riemannian Newton method \citep{agarwal2021adaptive, zhang2018cubic} can be applied, offering better convergence performances. When an efficient projection onto the manifold is available, \citet{ding2024convergence} proposed a projection-based framework specifically tailored for compact matrix manifolds. Additionally, stochastic variants of Riemannian gradient descent have been developed to improve scalability in large-scale settings \citep{bonnabel2013stochastic, han2021improved, hosseini2020recent, zhang2016riemannian}. 

For composite problems involving a nonsmooth convex term \(g\), \citet{chen2020proximal} proposed a Riemannian proximal gradient method over the Stiefel manifold, naturally extending the Euclidean proximal gradient framework to the manifold setting. \citet{huang2022riemannian, huang2023inexact} further generalized this approach to arbitrary Riemannian manifolds and introduced accelerated and inexact variants. Newton-type methods for nonsmooth composite Riemannian optimization problems have also been studied extensively; see, e.g., \citet{grohs2016nonsmooth}, \citet{hu2024projected}, \citet{si2024riemannian}, \citet{wang2023proximal}, and \citet{wang2024adaptive}. Recently, there have also been works extending Riemannian optimization methods to more complex settings, including minimax, bilevel, and zeroth-order optimization \citep{han2023nonconvex, han2023riemannian, han2024framework, he2024riemannian, li2023stochastic, li2025riemannian, zhang2023sion}.

\paragraph{Relatively smooth optimization.} The concept of relative smoothness was initially proposed by \citet{bauschke2017descent,lu2018relatively} in the context of convex optimization, relaxing the standard assumption of global gradient Lipschitz continuity. This notion motivates Bregman-type gradient methods, in which the traditional Euclidean distance is replaced by the Bregman divergence, thereby enabling them to accommodate a broader class of optimization problems. For instance, objective functions arising in the Poisson inverse problem and the D-optimal design problem have been shown to be relatively smooth with respect to the reference function \(h(x) = -\sum_{i=1}^n \log(x_i)\). More applications can be found in \citet{mukkamala2022bregman}.

Variants of Bregman-type gradient methods have also been well studied. \cite{hanzely2021accelerated, hanzely2021fastest} proposed accelerated Bregman proximal gradient methods for relatively smooth convex optimization, and \cite{takahashi2024approximate} developed an inexact version, which employs a new formulation that approximates the Bregman distance, making the subproblem easier to solve. Stochastic Bregman gradient methods have been studied in \cite{fatkhullin2024taming, ding2025nonconvex, dragomir2021fast}, demonstrating that relative smoothness can be applied to deep learning and differentially private optimization. For compact convex constraint sets, Frank–Wolfe methods under relative smoothness have been recently introduced and analyzed theoretically in \cite{vyguzov2024adaptive,takahashi2025fast}. 

\subsection{Main contributions}
In this paper, we extend the methodology of relatively smooth minimization to the setting of Riemannian optimization and introduce two Riemannian Bregman gradient methods. In Section \ref{section.retraction}, we propose the retraction-based Riemannian Bregman gradient method (Algorithm~\ref{alg.relative GD}) for nonsmooth optimization, which generalizes the update formulation of ManPG \citep{chen2020proximal} by employing the Bregman distance. Consequently, at each iteration, we solve a convex optimization subproblem involving the reference function and subject to a tangent space constraint. We show that for the quartic reference function $h(x) = \frac{1}{4}\|x\|^4 + \frac{1}{2}\|x\|^2$, the subproblem admits a closed-form solution (Proposition \ref{propo.quartic closed form}). Moreover, when the manifold is a sphere and the reference function is chosen to be either the log-barrier function or the entropy function, the subproblem simplifies to solving a one-dimensional nonlinear equation  (Proposition \ref{propo.sphere closed form}). By viewing the subproblem as a parametric optimization problem, we prove convergence to a stationary point on smooth, complete Riemannian embedded submanifolds and establish an iteration complexity of $\mathcal{O}(1 / \epsilon^2)$ for finding an $\epsilon$-approximate Riemannian stationary point of problem \eqref{eq.main} (Theorem~\ref{thm.convergence of retraction based}). 

In Section \ref{section.projection}, we develop the projection-based Riemannian Bregman gradient method (Algorithm~\ref{alg.projection relative GD}) for smooth Riemannian optimization ($g \equiv 0$ in \eqref{eq.main}). In this case, the subproblem becomes an easier unconstrained convex optimization problem. After obtaining the update direction, we project directly onto the manifold with an appropriate stepsize. By using projection-related inequalities (Lemmas \ref{lemma.projection inequality} and \ref{lemma.bound projection normal space}), we similarly establish convergence and iteration complexity (Theorem \ref{thm.convergence of projection based}). Interestingly, we find that the projection-based method generates the same update direction as the retraction-based method when using the quartic reference function over fixed-rank manifolds (Proposition \ref{prop.projection fixed rank}). In Section \ref{section.stochastic}, for compact submanifolds, we further develop corresponding stochastic variants and establish their sample complexity guarantees; see Theorems~\ref{thm.sample complexity relative} and~\ref{thm.sample complexity projection}, respectively. Numerical results on the nonlinear eigenvalue problem \eqref{eq.ne} and low-rank quadratic sensing problem \eqref{eq.low-rank recover} in Section \ref{section.numerics} demonstrate the efficiency of our Riemannian Bregman gradient methods.

\section{Preliminaries}
Throughout this paper, we use lowercase letters (e.g., $x, y, z$) to denote vectors and uppercase letters (e.g., $X, Y, Z$) to denote matrices. Unless otherwise specified, we use lowercase symbols in the main text. A function is said to be $C^k$ if it is $k$-times continuously differentiable; in particular, a $C^\infty$ function is said to be smooth. In this section, we provide a brief introduction to optimization over Riemannian manifolds. For more details, we refer the interested reader to textbooks \citep{absil2009optimization, boumal2023introduction}. We first introduce the notion of a differentiable submanifold via the implicit‐function theorem:
\begin{definition}[Differentiable submanifolds]\label{def.submanifold}
   A subset $\mathcal M\subseteq\mathbb R^n$ is called a
  $d$‑dimensional embedded $C^k$ submanifold, $k\ge1$, if for every
  $x\in\mathcal M$ there exist an open neighborhood
  $\mathcal U_x\subseteq\mathbb R^n$ and a $C^k$ map $\phi_x: \mathcal{U}_x \to \reals^{n-d}$ such that $\mathcal{U}_x \cap \mathcal{M} = \{y\in\mathcal{U}_x : \phi_x(y)=0\}$, and $\operatorname{rank}(\operatorname{J}\phi_x(y))=n-d$ for all $y\in \mathcal{U}_x\cap\mathcal{M}$, where $\operatorname{J}\phi_x$ denotes the Jacobian matrix of $\phi_x$.
\end{definition}

For a submanifold $\mathcal{M} \subseteq \reals^n$, the tangent space at point $x \in \mathcal{M}$, denoted by $\mathcal{T}_x \mathcal{M}$, can be characterized as a linear subspace of $\reals^n$ given by:
\begin{definition}[Tangent space]
    Given a submanifold $\mathcal{M} \subseteq \reals^n$, the tangent space of $\mathcal{M}$ at $x \in \M$ is defined as
    $$
    \mathcal{T}_x \mathcal{M}=\left\{\gamma^{\prime}(0): \gamma \text { is a smooth curve with } \gamma(0)=x, \ \gamma([-\delta, \delta]) \subseteq \mathcal{M}, \ \text {for some }\delta>0\right\}
    $$
    Consequently, the tangent bundle is defined as $\mathcal{T} \mathcal{M}=\left\{(x, \xi): x \in \mathcal{M}, \xi \in \mathcal{T}_x \mathcal{M}\right\}$. The normal space of $\mathcal{M}$ at $x$, denoted by $\N_x\M$, is the orthogonal complement to the tangent space $\mathcal{T}_x \mathcal{M}$. 
\end{definition}

For example, one commonly encountered submanifold is the Stiefel manifold, defined as $\operatorname{St}(m,p) = \{X \in \reals^{m \times p}: X^\top X - I_p = 0\}$, and the tangent space to $\operatorname{St}(m,p)$ at $X$ is $\T_X\operatorname{St}(m,p) = \{V \in \reals^{m \times p}: X^\top V + V^\top X = 0\}$. By the implicit function theorem, a useful result is that the tangent space can be expressed in terms of the Jacobian of some equations.
\begin{corollary}\label{coro.tangent space charactirization}
    Let $\mathcal{M}$ be a $d$‐dimensional embedded $C^k$ submanifold of $\reals^n$. Given $x\in\mathcal{M}$, for any $y\in \mathcal{U}_x\cap\mathcal{M}$, it holds that $\mathcal{T}_y \mathcal{M}=\operatorname{ker}(\operatorname{J}\phi_x(y)) = \{v \in \reals^n: \operatorname{J}\phi_x(y) v = 0\}$, where $\phi_x(\cdot)$ is defined in Definition \ref{def.submanifold}.
\end{corollary}

\begin{definition}[Riemannian (embedded) manifold]
    Let $\mathcal{M}$ be a differentiable submanifold of $\reals^n$. We say $\mathcal{M}$ is a Riemannian submanifold of $\reals^n$, if for any $x \in \mathcal{M}$ the tangent space $\mathcal{T}_x \mathcal{M}$ is endowed with a smooth inner product mapping $\langle\cdot, \cdot\rangle_x: \mathcal{T} \mathcal{M}\times \mathcal{T} \mathcal{M} \to \reals$; that is, for any $\eta, \xi \in \mathcal{T}_x \mathcal{M}$, $\langle\xi, \eta\rangle_x$ forms an inner product on $\mathcal{T}_x \mathcal{M}\times \mathcal{T}_x \mathcal{M}$. Denote the induced norm $\|\eta\|_x=\sqrt{\langle\eta, \eta\rangle_x}$ for any $\eta \in \mathcal{T}_x \mathcal{M}$.

    Further, we say $\mathcal{M}$ is a Riemannian embedded submanifold of $\reals^n$, if for any $x \in \mathcal{M}$, the tangent space $\mathcal{T}_x \mathcal{M}$ is endowed with the Euclidean inner product; that is, for any $\eta, \xi \in \mathcal{T}_x \mathcal{M}$, $\langle\xi, \eta\rangle_x \triangleq \langle\xi, \eta\rangle$, where the latter is the standard Euclidean inner product. Hence, the norm $\|\cdot\|_x$ is the same as the standard $\ell_2$-norm or the Frobenius norm in the matrix case.
\end{definition}

Without loss of generality, we omit the subscript $x$ in the inner product $\langle\cdot, \cdot\rangle$ and the norm $\|\cdot\|$ since our focus is on Riemannian embedded submanifolds. For any point \(x \in \mathbb{R}^n\) and a nonempty subset \(\mathcal{S} \subseteq \mathbb{R}^n\), we denote by \(\mathcal{P}_{\mathcal{S}}(x)\) the projection of \(x\) onto \(\mathcal{S}\) if it exists. We use \(\overline{\mathcal{S}}\) to denote the closure of the set \(\mathcal{S}\). The open Euclidean ball of radius \(r > 0\) centered at \(x \in \mathbb{R}^n\) is denoted by $\mathbb{B}(x, r) \triangleq \{y \in \reals^n: \|y-x\| < r\}$. In this setting, the Riemannian gradient is defined as the projection of the Euclidean gradient onto the tangent space of the manifold. 
\begin{definition}[Riemannian gradient]
   Let $f$ be a continuously differentiable function on $\reals^n$. The Riemannian gradient $\operatorname{grad} f(x)$ of $f$ with respect to a submanifold $\mathcal{M}$ is a tangent vector in $\mathcal{T}_x \mathcal{M}$ defined by
    $$
    \grad f(x)=\operatorname{\P}_{\mathcal{T}_x \mathcal{M}}(\nabla f(x)),
    $$
    where $\P_{\T_x\M}$ is the orthogonal projection onto the tangent space $\T_x\M$. We say $x$ is a Riemannian stationary point of the differentiable function $f$ if it satisfies $\grad f(x) = 0$.
\end{definition}

For a convex function $g$, its Euclidean subgradient at point $x$ is denoted by $\partial g(x)$. Similarly, the  Riemannian subgradient is defined  as $\hat{\partial} g(x) = \operatorname{\P}_{\mathcal{T}_x \mathcal{M}}(\partial g(x))$. From \cite{chen2020proximal}, the optimality condition for problem \eqref{eq.main} is given as follows:
\begin{definition}[Optimality condition]
    A point $x \in \M$ is called a Riemannian stationary point of problem \eqref{eq.main} if it satisfies $0 \in \grad f(x) + \operatorname{\P}_{\mathcal{T}_x \mathcal{M}}(\partial g(x))$.
\end{definition}

A key ingredient in Riemannian optimization is the notion of a retraction, which is a first-order approximation of the exponential mapping and is often more amenable for computation. Its formal definition is given below.
\begin{definition}[Retraction]
   A retraction on a manifold $\mathcal{M}$ is a smooth mapping Retr from the tangent bundle $\T \mathcal{M}$ to $\mathcal{M}$ with the following properties. Let $\retr(x, \cdot): \T_x \mathcal{M} \to \mathcal{M}$ denote the restriction of Retr to $\T_x \mathcal{M}$.
   \begin{enumerate}
       \item $\retr\left(x, 0_x\right)=x$, where $0_x$ is the zero vector in $\T_x \mathcal{M}$;
       \item The differential of $\retr(x, \cdot)$ at $0_x$, i.e., $\operatorname{D}\retr\left(x, 0_x\right)$, is the identity map. 
   \end{enumerate}
\end{definition}

When the manifold is complete, the domain of the retraction is the entire tangent bundle. By the smoothness of the retraction, for any $(x, v) \in \T\M$, there exist constants $M_1^{\operatorname{R}}(x,v), M_2^{\operatorname{R}}(x,v) \ge 0$ such that
\begin{equation}\label{eq.retraction inequality}
    \begin{aligned}
        \|\retr(x, v) - x\| \ &\le \ M_1^{\operatorname{R}}(x,v) \|v\|, \\
        \|\retr(x, v) - (x + v)\| \ &\le \ M_2^{\operatorname{R}}(x,v) \|v\|^2,
    \end{aligned}
\end{equation}
where $M_1^{\operatorname{R}}(x,v) = \max_{\xi \in \overline{\mathbb{B}}(x, \|v\|)}\|\operatorname{D}\retr(x, \xi)\|$, and $M_2^{\operatorname{R}}(x,v) = \max_{\xi \in \overline{\mathbb{B}}(x, \|v\|)}\|\operatorname{D}^2\retr(x, \xi)\|$. These inequalities follow directly from Lemma 4 in \cite{boumal2019global}. However, these two constants are no longer uniform since we do not restrict our analysis to compact manifolds. Specifically, $M_1^{\operatorname{R}}(x,v)$ and $M_2^{\operatorname{R}}(x,v)$ depend on both the current point $x$ and tangent vector $v$. We close this section by stating the following assumptions used throughout the paper. These conditions are standard in Riemannian optimization (see, e.g., \cite{chen2020proximal, zhang2016first}).
\begin{assumption}\label{assumption}
    In problem \eqref{eq.main}, the Riemannian embedded submanifold $\M$ is $C^\infty$ and complete. The smooth part $f$ is $L$-smooth relative to a reference function $h$, where $h$ is continuously differentiable and $\lambda$-strongly convex with $\lambda > 0$. The nonsmooth term $g$ is $L_g$-Lipschitz continuous. The sublevel set of function $F$ at some point $\widetilde{x}$ is compact, i.e., $\mathcal{L}(\widetilde{x}) \triangleq \{x \in \M: F(x) \le F(\widetilde{x})\}$ is compact.
\end{assumption}

\section{Retraction-Based Riemannian Bregman Gradient Method}\label{section.retraction}


Most optimization methods over Riemannian manifolds share a common update scheme: they first solve a subproblem in the tangent space of the current iterate, which returns a suitable descent direction; then use a retraction along this direction with an appropriate stepsize to obtain the next iterate. This idea is natural and allows Riemannian optimization methods to mimic their Euclidean counterparts. In this section, we develop a Riemannian Bregman gradient method by following this update paradigm.

Given a point $x \in \M$, we use the Bregman distance induced by a reference function $h$ to guide the update direction. Specifically, we solve the following subproblem in the tangent space $\T_x\M$:
\begin{equation}\label{eq.update grad}
    v^*(x) \ = \ \argmin_{v \in \T_{x}\M} \ \inner{\grad f(x)}{v} + \gamma D_h(x + v, x) + g(x + v),
\end{equation}
where $\gamma > 0$ can be viewed as the stepsize. If $h(x) = \frac{1}{2}\|x\|^2$, i.e., the Euclidean squared norm, and $\M$ is the Stiefel manifold, then the update rule in \eqref{eq.update grad} reduces to the ManPG proposed in \cite{chen2020proximal}. 
According to Theorem 4.1 in \cite{yang2014optimality}, the first-order optimality condition for the subproblem \eqref{eq.update grad} is characterized by
\begin{align*}
    0 \ \in \ \grad f(x) + \gamma \cdot \P_{\T_{x}\M}\left(\nabla h(x + v^*(x)) - \nabla h(x)\right) + \P_{\T_{x}\M} \left(\partial g(x + v^*(x))\right).
\end{align*}
Thus, if \( v^*(x) = 0 \), the condition reduces to \( 0 \in \grad f(x) + \P_{\T_{x}\M} \left(\partial g(x) \right)\), which is exactly the optimality condition of problem \eqref{eq.main}. Hence, the magnitude of direction $v^*(x)$ can be viewed as a stationary measure. Since the subproblem is restricted to the tangent space, we have $\inner{\grad f(x)}{v} = \inner{\nabla f(x)}{v}, \ v \in \T_{x}\M$, by the definition of the Riemannian gradient. It is therefore unnecessary to compute the Riemannian gradient \( \grad f(x) \); instead, the Euclidean gradient \( \nabla f(x) \) suffices. After obtaining the direction \( v^*(x) \), one can choose a suitable stepsize via backtracking linesearch with a shrinkage parameter. Combining the above components, we summarize our retraction-based method for solving \eqref{eq.main} in Algorithm~\ref{alg.relative GD}. During the iterations of this algorithm, we terminate once the norm of the update direction $\|v_t\|$ (Line 3 in Algorithm~\ref{alg.relative GD}) becomes small. Specifically, we define the $\epsilon$-approximate stationary point as follows.
\begin{definition}\label{eq.stationary point retraction}
   Given accuracy $\epsilon > 0$, we say $x_t$ is an $\epsilon$-approximate Riemannian stationary point of problem \eqref{eq.main} whenever $\|v_t\| \le \epsilon$, where $v_t$ is defined in \eqref{eq.update nabla}. 
\end{definition}

\begin{algorithm}[h]
	\caption{Retraction-Based Riemannian Bregman Gradient Method}
	\label{alg.relative GD}
	\begin{algorithmic}[1] 
        \item \textbf{Input:} initial point $x_0 \in \M$, $\gamma_t \ge L$, $\rho \in (0,1)$
        \item \textbf{For} $t = 0, 1, \ldots$ \textbf{do}
        \item \quad \quad Obtain $v_t$ by solving the subproblem
        \begin{equation}\label{eq.update nabla}
             v_t \ = \ \argmin_{v \in \T_{x_t}\M} \ \inner{\nabla f(x_t)}{v} + \gamma_t D_h(x_t + v, x_t) + g(x_t + v)
        \end{equation}
         \item \quad \quad Set the initial stepsize $\alpha_t = 1$
         \item  \quad \quad \textbf{While} $F\left(\retr(x_t,\alpha_t v_t)\right) - F(x_t) \ > \ -\frac{\gamma_t\lambda\alpha_t}{4} \|v_t\|^2$ \textbf{do}
          \item  \quad \quad \quad \quad $\alpha_t := \rho \alpha_t$
          \item  \quad \quad \textbf{ end While}
           \item \quad \quad Update $x_{t+1} = \retr(x_t,\alpha_t v_t)$
	\end{algorithmic}
\end{algorithm}

Notice that the next iterate \( x_{t+1} \) lies on the manifold \( \M \), whereas the subproblem~\eqref{eq.update nabla} is approximately solved in the tangent space \( \T_{x_t}\M \), this induces an approximation error. Let \( x_t^+ \triangleq x_t + \alpha_t v_t \) denote the intermediate point in the tangent space. The following lemma provides an upper bound on the discrepancy between \( D_h(x_{t+1}, x_t) \) and \( D_h(x_t^+, x_t) \). As the analysis is localized around the iterate \( x_t \), we introduce the radius \( r(x,v) \triangleq M_1^{\operatorname{R}}(x,v)\|v\| \), where \( M_1^{\operatorname{R}}(x,v) \) is the constant from inequality~\eqref{eq.retraction inequality}.

\begin{lemma}\label{lemma.bound bregman deviation}
     Suppose Assumption \ref{assumption} holds. Let $v_t$ be the solution to \eqref{eq.update nabla}. For any $\alpha_t \in (0, 1)$, it holds that $D_h(x_t^+, x_t) \le \alpha_t D_h(x_t + v_t, x_t)$ and 
     $$
     D_h(x_{t+1}, x_t) - D_h(x_t^+, x_t) \le 2G_h(x_t, r(x_t, v_t)) M_2^{\operatorname{R}}(x_t,v_t)\|\alpha_t v_t\|^2,
     $$ 
     where $G_h(x_t, r(x_t,v_t))$ is defined in \eqref{eq.G_h(x,r(x, v)) retraction}.
\end{lemma}
\begin{proof}
From the definition of the Bregman distance and the convexity of \( h \), we obtain
\begin{align*}
      D_h(x_t^+, x_t) \ &= \ h(x_t + \alpha_t v_t) - h(x_t) - \inner{\nabla h(x_t)}{\alpha_t v_t} \\
      &= \ h\left(\alpha_t(x_t + v_t) + (1-\alpha_t)x_t\right) - h(x_t) - \inner{\nabla h(x_t)}{\alpha_t v_t} \\
      &\le \ \alpha_t h(x_t + v_t) + (1-\alpha_t)h(x_t) - h(x_t) - \inner{\nabla h(x_t)}{\alpha_t v_t} \\
      &= \ \alpha_t h(x_t + v_t) - \alpha_t h(x_t) - \inner{\nabla h(x_t)}{\alpha_t v_t} \\
      &= \ \alpha_t D_h(x_t + v_t, x_t).
\end{align*}
Next, we bound the error between $D_h(x_{t+1}, x_t)$ and $D_h(x_t^+, x_t)$. By the retraction inequality \eqref{eq.retraction inequality}, it follows
    \begin{align*}
        \|x_{t+1} - x_t\| = \|\retr(x_t,\alpha_t v_t) - x_t\| \le \alpha_t M_1^{\operatorname{R}}(x_t,v_t) \|v_t\| \le M_1^{\operatorname{R}}(x_t,v_t) \|v_t\|,
    \end{align*}
which implies that \(x_{t+1} \in \overline{\mathbb{B}}(x_t, r(x_t, v_t)) \). Let 
\begin{equation}\label{eq.G_h(x,r(x, v)) retraction}
    G_h(x,r(x, v)) \triangleq \max_{y \in \overline{\mathbb{B}}(x, r(x, v))} \|\nabla h(y)\|
\end{equation}
denote the maximum gradient norm of the reference function \( h \) over the Euclidean ball centered at \( x \) with radius \( r(x, v) \). Then we have
    \begin{align*}
       &D_h(x_{t+1}, x_t) - D_h(x_t^+, x_t) \\
        = \ &h(x_{t+1}) - h(x_t^+) - \inner{\nabla h(x_t)}{x_{t+1} -x_t^+} \\
        \le \ &\inner{\nabla h(x_{t+1})}{x_{t+1} - x_t^+} + \|\nabla h(x_t)\| \cdot \|x_{t+1} - x_t^+\| \\
        \le \ &2G_h(x_t, r(x_t, v_t)) \|x_{t+1} - x_t^+\| \\
        \le \ &2G_h(x_t, r(x_t, v_t)) M_2^{\operatorname{R}}(x_t,v_t) \|\alpha_t v_t\|^2 
    \end{align*}
    where the last inequality comes from \eqref{eq.retraction inequality}. 
\end{proof}

We now present the per-iteration descent lemma, which ensures a sufficient decrease in the function value for a small enough stepsize.

\begin{lemma}\label{lemma.descent property}
    Suppose Assumption \ref{assumption} holds. For any $\gamma_t \ge L$, there exists a constant $\alpha_t^\prime > 0$ such that for any $0 < \alpha_t \le \min\{1, \alpha_t^\prime\}$, the next iterate $x_{t+1}$ in Algorithm \ref{alg.relative GD} satisfies
   \begin{align*}
       F\left(x_{t+1}\right) - F(x_t) \ \le \ -\frac{\gamma_t\lambda\alpha_t}{4} \|v_t\|^2,
   \end{align*}
   where $\alpha_t^\prime$ is defined in \eqref{eq.alpha bar retraction}.
\end{lemma}
\begin{proof}
    By the relatively $L$-smooth of $f$ and $\gamma_t \ge L$, it holds that
    \begin{align*}
        &f(x_{t+1}) - f(x_t) \\
        \le \ &\inner{\nabla f(x_t)}{x_{t+1} - x_t} + \gamma_t D_h(x_{t+1}, x_t) \\
        = \ &\inner{\nabla f(x_t)}{x_{t+1} - x_t^+ + x_t^+ - x_t} + \gamma_t D_h(x_t^+, x_t) + \gamma_t D_h(x_{t+1}, x_t) - \gamma_t D_h(x_t^+, x_t) \\
        \le \ &\inner{\nabla f(x_t)}{x_{t+1} - x_t^+ + x_t^+ - x_t} + \gamma_t \alpha_t D_h(x_t + v_t, x_t) + 2\gamma_t G_h(x_t, r(x_t, v_t)) M_2^{\operatorname{R}}(x_t,v_t) \|\alpha_t v_t\|^2.
    \end{align*}
    We use Lemma \ref{lemma.bound bregman deviation} in the last inequality. For the inner product term, using the retraction property \eqref{eq.retraction inequality} yields
    \begin{align*}
        &\inner{\nabla f(x_t)}{x_{t+1} - x_t^+ + x_t^+ - x_t} \\
        = \ &\inner{\nabla f(x_t)}{x_{t+1} - x_t^+} + \alpha_t\inner{\nabla f(x_t)}{v_t} \\
        \le \ &M_2^{\operatorname{R}}(x_t,v_t)\|\nabla f(x_t)\| \cdot \|\alpha_t v_t\|^2 + \alpha_t\inner{\nabla f(x_t)}{v_t}.
    \end{align*}
    Consequently, we obtain
   \begin{align*}
       &f(x_{t+1}) - f(x_t) \\
       \le \ &\alpha_t\left(\inner{\nabla f(x_t)}{v_t} + \gamma_t D_h(x_t + v_t, x_t)\right) + \left(\|\nabla f(x_t)\| + 2\gamma_tG_h(x_t, r(x_t, v_t))\right)M_2^{\operatorname{R}}(x_t,v_t)\|\alpha_t v_t\|^2.
   \end{align*}
   Since $\M$ is embedded in the Euclidean space $\reals^n$, the tangent space $\T_{x_t}\M$ is closed and convex. Due to the optimality condition of constrained optimization, it follows
   \begin{align*}
      \inner{\nabla f(x_t) + \gamma_t\nabla h(x_t + v_t) - \gamma_t \nabla h(x_t) + s_t}{v - v_t} \ge 0, \ \forall v \in \T_{x_t}\M,  
   \end{align*}
   where $s_t \in \partial g(x_t + v_t)$. Specifically, choose \( v \) to be the zero vector in \( \T_x \mathcal{M} \); this yields
   \begin{align*}
       \inner{\nabla f(x_t) - \gamma_t\nabla h(x_t)}{v_t} \le \inner{\gamma_t \nabla h(x_t + v_t)}{-v_t} - \inner{s_t}{v_t}.
   \end{align*}
   Hence, we have
   \begin{align*}
       &\alpha_t\inner{\nabla f(x_t)}{v_t} + \gamma_t D_h(x_t^+, x_t) \\
        \le \ &\alpha_t \left[\inner{\nabla f(x_t)}{v_t} + \gamma_t D_h(x_t + v_t, x_t) \right] \\
       = \ &\alpha_t \left[\inner{\nabla f(x_t)}{v_t} + \gamma_t h (x_t + v_t) - \gamma_t h(x_t) - \gamma_t \inner{\nabla h(x_t)}{v_t} \right] \\
       = \ &\alpha_t \left[\inner{\nabla f(x_t) - \gamma_t \nabla h(x_t)}{v_t} + \gamma_t h (x_t + v_t) - \gamma_t h(x_t) \right] \\
       \le \ &\alpha_t \left[\inner{\gamma_t \nabla h(x_t + v_t)}{-v_t} + \gamma_t h (x_t + v_t) - \gamma_t h(x_t) \right] -\alpha_t \inner{s_t}{v_t}\\
       = \ &-\alpha_t \gamma_t \left[h(x_t) - h(x_t + v_t) - \inner{\nabla h(x_t + v_t)}{-v_t}\right] -\alpha_t \inner{s_t}{v_t}\\
       \le \ & -\frac{\alpha_t \gamma_t \lambda}{2}\|v_t\|^2 -\alpha_t \inner{s_t}{v_t}.
   \end{align*}
   Therefore, the descent property of smooth part can be established as follows:
    \begin{align*}
       f(x_{t+1}) - f(x_t) \ &\le \ -\frac{\alpha_t \gamma_t \lambda}{2}\|v_t\|^2 +  \left(\|\nabla f(x_t)\| + 2\gamma_tG_h(x_t, r(x_t, v_t))\right)M_2^{\operatorname{R}}(x_t,v_t)\|\alpha_t v_t\|^2 -\alpha_t \inner{s_t}{v_t}\\
       &= \ \left( \left(\|\nabla f(x_t)\| + 2\gamma_tG_h(x_t, r(x_t, v_t))\right)M_2^{\operatorname{R}}(x_t,v_t) - \frac{\gamma_t \lambda}{2\alpha_t}\right)\|\alpha_t v_t\|^2 -\alpha_t \inner{s_t}{v_t}.
   \end{align*}
   As for the nonsmooth part $g$, we have
   \begin{align*}
       g(x_{t+1}) - g(x_t) \ &= \ g(x_{t+1}) - g(x_t^+) + g(x_t^+) - g(x_t) \\
       &\le L_g\|x_{t+1} - x_t^+\| + \alpha_t \left(g(x_t + v_t) - g(x_t)\right) \\
        &\le L_g M_2^{\operatorname{R}}(x_t,v_t)\|\alpha_t v_t\|^2 + \alpha_t \inner{s_t}{v_t},
   \end{align*}
   where we use the $L_g$-Lipschitz continuity of $g$ and $g(x_t^+) = g(\alpha_t (x_t + v_t) + (1-\alpha_t)x_t) \le \alpha_t g(x_t + v_t) + (1-\alpha_t) g(x_t)$ in the second inequality, and the last inequality holds due to the convexity of $g$. Combining the decrease of smooth part and nonsmooth part yields
   \begin{align*}
       F(x_{t+1}) - F(x_t) \ \le \ \left( \left(\|\nabla f(x_t)\| + 2\gamma_tG_h(x_t, r(x_t, v_t)) + L_g\right)M_2^{\operatorname{R}}(x_t,v_t) - \frac{\gamma_t \lambda}{2\alpha_t}\right)\|\alpha_t v_t\|^2.
   \end{align*}
   By setting 
   \begin{equation}\label{eq.alpha bar retraction}
       \alpha_t^\prime \triangleq \frac{\gamma_t\lambda}{4 (\|\nabla f(x_t)\| + 2\gamma_t G_h(x_t, r(x_t, v_t)) + L_g)M_2^{\operatorname{R}}(x_t,v_t)},
   \end{equation}
   we conclude that for any $0 < \alpha_t \le \min\{1, \alpha_t^\prime\}$,
   \begin{align*}
       F(x_{t+1}) - F(x_t) \ \le \ -\frac{\gamma_t\lambda\alpha_t}{4} \|v_t\|^2.
   \end{align*}
   Thus, the proof is completed.
\end{proof}

The above lemma ensures that the while loop (Line 5) in Algorithm~\ref{alg.relative GD} is well-defined and terminates in a finite number of steps. It also guarantees that $x_{t+1} \in \mathcal{L}(\widetilde{x})$ whenever $x_{t} \in \mathcal{L}(\widetilde{x})$. However, the above descent lemma is a local result: since we do not assume the manifold \( \M \) to be compact, constants such as the stepsize $\alpha_t$ depends on the current iterate \( x_t \) and the update direction \( v_t \). If the stepsize sequence $\{\alpha_t\}_{t \ge 0}$ admits a uniform lower bound across iterations, we readily obtain a convergence result of Algorithm \ref{alg.relative GD} from above lemma. To establish such a uniform lower bound, we view the subproblem \eqref{eq.update grad} as a parametric optimization problem, and then show that the solution $v^*(x)$ is a continuous function. The proof requires standard concepts from variational analysis, which are provided in the Appendix (see Definitions \ref{def.set valued map} and \ref{def.continuity of set valued map}).

\begin{lemma}\label{lemma.continuity of v*}
Suppose \(\varphi: \mathbb{R}^n\times\mathbb{R}^n\to\mathbb{R}\) is jointly continuous in \((x,v)\), and for each fixed \(x\), the function \(v\mapsto\varphi(x,v)\) is \(\lambda\)-strongly convex with \(\lambda>0\). Let \(S\colon\mathbb{R}^n\rightrightarrows\mathbb{R}^n\) be a continuous set-valued map such that, for every \(x\), \(S(x)\) is a linear subspace of \(\mathbb{R}^n\). Then the minimizer $v^*(x) = \arg\min_{v\in S(x)} \varphi(x,v)$ defines a continuous function \(v^*: \mathbb{R}^n\to\mathbb{R}^n\).
\end{lemma}
\begin{proof}
Since $\varphi(x, \cdot)$ is $\lambda$–strongly convex for any fixed $x$, its restriction to $S(x)$ remains $\lambda$–strongly convex and hence coercive. Consequently, the minimum $v^*(x)$ on the subspace $S(x)$ exists and, by strong convexity, is unique. Notice that $\varphi(x, v^*(x)) \leq \varphi(x, 0)$, and $\varphi(x, v^*(x)) \geq \varphi(x, 0) + \langle s_0(x), v^*(x) \rangle + \frac{\lambda}{2} \|v^*(x)\|^2$, where $s_0(x) \in \partial_v \varphi(x, 0)$. Combining these, we get $0 \geq \langle s_0(x), v^*(x) \rangle + \frac{\lambda}{2} \|v^*(x)\|^2 \ge -\|s_0(x)\| \cdot \|v^*(x)\| + \frac{\lambda}{2} \|v^*(x)\|^2$. Suppose $v^*(x) \neq 0$, then we have $\|v^*(x)\| \leq 2\|s_0(x)\| / \lambda$; otherwise $v^*(x) = 0$, $\|v^*(x)\| \leq 2\|s_0(x)\| / \lambda$ stills holds. 

Fix $x$, we claim that $s_0(\cdot)$ is locally bounded around $x$. Otherwise, there exists a sequence $x_k \to x$ and $s_0(x_k) \in \partial_v \varphi(x_k, 0)$ such that $\|s_0(x_k)\| \to \infty$. Let $\zeta(x_k) \triangleq s_0(x_k) / \|s_0(x_k)\|^2$. Clearly, $\inner{\zeta(x_k)}{s_0(x_k)} = 1$ and $\|\zeta(x_k)\| = 1 / \|s_0(x_k)\| \to 0$. By the strong convexity of $\varphi(x_k, \cdot)$, it follows $\varphi(\zeta(x_k), x_k) \ge \varphi(x_k, 0) + \inner{s_0(x_k)}{\zeta(x_k)} + \frac{\lambda}{2}\|\zeta(x_k)\|^2 \ge \varphi(x_k, 0) + 1$. Let $k \to \infty$. Since $\varphi$ is jointly continuous, we have $\varphi(x, 0) \ge \varphi(x, 0) + 1$, which is a contradiction.

Choose a sequence $x_k \to x$ and set $v_k \triangleq v^*(x_k)$. By the above, the sequence $\{v_k\}$ is bounded, so it admits a convergent subsequence $v_{k_j} \to v^\prime$. Since $v_{k_j} \in S(x_{k_j})$ and $S$ is outer semicontinuous, it follows that $v^\prime \in S(x)$. For any $w \in S(x)$, by the inner semicontinuous property of $S$, there exist a sequence $w_j \in S(x_{k_j})$ with $w_j \to w$. By continuity of $\varphi$, we have $\varphi(x_{k_j}, v_{k_j}) \to \varphi(x, v^\prime)$ and $\varphi(x_{k_j}, w_j) \to \varphi(x, w)$. Since $v_{k_j}$ is the unique minimizer over $S(x_{k_j})$, we have $\varphi(x_{k_j}, v_{k_j}) \leq \varphi(x_{k_j}, w_j)$ for each $j$, hence $\varphi(x, v^\prime) \leq \varphi(x, w)$. Thus, $v^\prime$ is the unique minimizer of $\varphi(x, \cdot)$ over $S(x)$, i.e., $v^\prime = v^*(x)$. Therefore, any convergent subsequence of $\{v_k\}$ converges to $v^*(x)$. Since all such subsequences have the same limit, by contradiction we could conclude that $v_k \to v^*(x)$ as $k \to \infty$, i.e., $v^*(x)$ is continuous in $x$.
\end{proof}

For the subproblem \eqref{eq.update grad}, we choose $\varphi(x,v) = \inner{\grad f(x)}{v} + \gamma D_h(x + v, x) + g(x + v), \ \gamma > 0$. By Corollary \ref{coro.tangent space charactirization}, for any $x \in \M$, there exists a smooth mapping $\phi_x$ defined on an open neighborhood $\mathcal{U}_x$ of $x$, satisfying $\operatorname{rank}(\operatorname{J}\phi_x(y))=n-d$, and $\T_y\M = \operatorname{ker}(\operatorname{J}\phi_x(y))$ for any $y \in \mathcal{U}_x$. Clearly, the tangent space mapping $\T_y\M = \operatorname{ker}(\operatorname{J}\phi_x(y))$ is a continuous set-valued map on $\mathcal{U}_x$. Additionally, $\varphi$ is jointly continuous in $(x,v)$, and $\lambda$-strongly convex for each fixed $x$. Thus, by applying the lemma above, we conclude that the solution mapping $v^*(x)$ is continuous, which can be used to establish the following Theorem.

\begin{theorem}\label{thm.convergence of retraction based}
  Suppose Assumption \ref{assumption} holds. Set the initial point $x_0 = \widetilde{x}$. Then every limit point of the sequence $\left\{x_t\right\}_{t \ge 0}$ generated by Algorithm \ref{alg.relative GD} with $\gamma_t = L$ satisfies the optimality condition of problem \eqref{eq.main}. Moreover, for any given accuracy \( \epsilon > 0 \), after at most \( \mathcal{O}(\epsilon^{-2}) \) iterations, Algorithm~\ref{alg.relative GD} with \( \gamma_t = L \) returns a direction \( v_t \) satisfying \( \|v_t\| \le \epsilon \).
\end{theorem}
\begin{proof}
    We first argue that the previously defined constants \( M_1^{\operatorname{R}}(x,v) \), \( M_2^{\operatorname{R}}(x,v) \), and \( G_h(x, r(x,v)) \) are continuous in \( x \). Recall that
    \begin{align*}
        M_1^{\operatorname{R}}(x, v^*(x)) \ = \ \max_{\xi \in \overline{\mathbb{B}}(x, \|v^*(x)\|)}\|\operatorname{D}\retr(x, \xi)\|.
    \end{align*}
    Since $v^*(x)$ is continuous, then the set‐valued map $\overline{\mathbb{B}}(x, \|v^*(x)\|)$ varies continuously with \(x\). Clearly, for each $x$, $\overline{\mathbb{B}}(x, \|v^*(x)\|)$ is non-empty and compact. Hence, due to the smoothness of the retraction, it follows from Berge’s Maximum Theorem (cf. Theorem \ref{thm.Berge}) that \( M_1^{\operatorname{R}}(x, v^*(x)) \) is continuous in $x$. By the same argument, \( M_2^{\operatorname{R}}(x, v^*(x)) \) is also continuous in $x$. As a consequence, the radius \( r(x, v^*(x)) = M_1^{\operatorname{R}}(x, v^*(x))\|v^*(x)\| \) is continuous. By applying Berge’s Maximum Theorem again, and noting that $G_h(x, r(x, v^*(x)))$ defined in \eqref{eq.G_h(x,r(x, v)) retraction} is the maximum value over compact-valued continuous set-valued mapping, we conclude its continuity in $x$.
    
    Now we proceed to show the convergence. Without loss of generality, assume that \( v_t \neq 0 \) for all \( t \ge 0 \); otherwise, \( x_t \) is already a stationary point. By Lemma \ref{lemma.descent property}, we know the Algorithm \ref{alg.relative GD} is monotone, and the iterates \( \{x_t\}_{t \ge 0} \) remain within the sublevel set $\mathcal{L}(\widetilde{x})$ when $x_0 = \widetilde{x}$. Since $\mathcal{L}(\widetilde{x})$ is compact, it follows that the sequence \( \{v_t\}_{t \ge 0} \) is bounded. As a result, the sequences \( \{M_1^{\operatorname{R}}(x_t, v_t)\}_{t \ge 0} \), \( \{M_2^{\operatorname{R}}(x_t, v_t)\}_{t \ge 0} \), and \( \{G_h(x_t, r(x_t, v_t))\}_{t \ge 0} \) are bounded. Hence, there exist constants \(M_2^{\operatorname{R}} > 0\) and \(G_h > 0\) such that \( M_2^{\operatorname{R}}(x_t, v_t) \le M_2^{\operatorname{R}} \) and \( G_h(x_t, r(x_t, v_t)) \le G_h \) for all \( t \ge 0 \). Due to the backtracking linesearch, it holds that \( \alpha_t \ge \rho \alpha_t^\prime \). Then we can find a constant \( \alpha^\prime > 0 \) such that for all \( t \ge 0 \),
    \begin{align*}
        \alpha_t \ \ge \ &\frac{\rho \gamma_t \lambda}{4 \left( \|\nabla f(x_t)\| + 2G_h(x_t, r(x_t, v_t)) \gamma_t + L_g \right) M_2^{\operatorname{R}}(x_t, v_t)} 
        \\
        \ge \ &\frac{\rho L \lambda}{4 \left( G_f + 2G_h L + L_g \right) M_2^{\operatorname{R}}} \ \triangleq \ \alpha^\prime,
    \end{align*}
    where \( G_f \) is the upper bound for sequence \( \{\|\nabla f(x_t)\|\}_{t \ge 0} \). Since \( F \) is lower bounded on \( \M \), the decrease property established in Lemma~\ref{lemma.descent property} implies that \( \lim_{t \to \infty} \|v_t\| = 0 \). Therefore, the sequence \( \{x_t\}_{t \ge 0} \) converges to a stationary point of problem \eqref{eq.main}. Moreover, the compactness of $\mathcal{L}(x_0)$ yields that \(\{x_t\}_{t \ge 0}\) admits at least one limit point.  
    
    Finally, we analyze the iteration complexity. Suppose that Algorithm \ref{alg.relative GD} with $\gamma_t = L$ does not terminate after $T$ iterations; i.e., $\|v_t\| > \epsilon / L$ for all $t=0,1, \ldots, T-1$, then it follows 
     \begin{align*}
       F(x_{t+1}) - F(x_t) \ \le \ -\frac{L\lambda\alpha^\prime}{4} \|v_t\|^2 \ \le \ - \frac{\lambda\alpha^\prime}{4 L} \epsilon^2.
   \end{align*}
    Notice that $F^* - F(x_0) \le F(x_t) - F(x_0) = \sum_{t=0}^{T-1} [F(x_{t+1}) - F(x_t)]$, and we conclude $T = \mathcal{O}(\epsilon^{-2})$. The proof is completed.
\end{proof}

\begin{remark}
    When the reference function is chosen as \( h(x) = \frac{1}{2}\|x\|^2 \) and \( \M \) is the Stiefel manifold, the iteration complexity result in Theorem~\ref{thm.convergence of retraction based} recovers the result established in \cite{chen2020proximal}. Our analysis generalizes their result by incorporating a Bregman distance framework and allowing for general Riemannian embedded submanifolds. Moreover, we only assume that the sublevel set is bounded, which is a mild and commonly used condition in complexity analysis. 
\end{remark}

The subproblem \eqref{eq.update nabla} in Algorithm~\ref{alg.relative GD} is a strongly convex minimization problem over a linear subspace. It can be efficiently solved via projected gradient descent or the regularized semismooth Newton method proposed in \cite{chen2020proximal}. In the smooth setting, i.e., $g \equiv 0$ in \eqref{eq.main}, by substituting the definition of the Bregman distance, the subproblem~\eqref{eq.update nabla} reduces to 
\begin{equation}\label{eq.simplified subproblem}
    \min_{v \in \mathcal{T}_{x_t}\M} \ \inner{c_t}{v} + h(x_t +v),
\end{equation}
where $c_t \triangleq \nabla f(x_t) / \gamma_t - \nabla h(x_t)$. Next, we prove that, with the quartic reference function \( h(x) = \frac{1}{4}\|x\|^4 + \frac{1}{2}\|x\|^2 \), the update direction $v_t$ admits an explicit closed-form expression on any Riemannian embedded submanifold. Besides, when the manifold is a sphere $\mathbb{S}^{n-1} = \{x \in \reals^n: \|x\| = 1\}$ and the reference function is chosen to be either the log-barrier function or the entropy function, the corresponding subproblem reduces to solving a one-dimensional nonlinear equation, which is similar to its Euclidean counterpart (see Eq. (18) in \cite{lu2018relatively}). Hence, one can use the bisection method, Newton’s method, or any other suitable scalar root-finding method to efficiently compute the solution.

\begin{proposition}\label{propo.quartic closed form}
    Suppose the nonsmooth term $g \equiv 0$, and consider the reference function $h(x) = \frac{1}{4}\|x\|^4 + \frac{1}{2}\|x\|^2$. Then the solution to subproblem \eqref{eq.update nabla} admits the closed form $v_t = -\theta_t \cdot \P_{\mathcal{T}_{x_t}\M}\left(\nabla f(x_t) / \gamma_t - \nabla h(x_t)\right) - \P_{\mathcal{T}_{x_t}\M}\left(x_t\right)$, where $\theta_t > 0$ is the unique positive solution to equation \eqref{eq.one dimensional equation}.
\end{proposition}
\begin{proof}
    First notice that the subproblem solution $v_t$ satisfies $\P_{\mathcal{T}_{x_t}\M}\left(c_t + \nabla h(x_t + v_t)\right) = 0$ due to the optimality condition of \eqref{eq.simplified subproblem}. When $h(x) = \frac{1}{4}\|x\|^4 + \frac{1}{2}\|x\|^2$, it implies that
     \begin{align*}
        \P_{\mathcal{T}_{x_t}\M}\left(c_t\right) +(\|x_t + v_t\|^2 + 1)\left(\P_{\mathcal{T}_{x_t}\M}\left(x_t\right)+ v_t\right) = 0.
    \end{align*}
    If $\P_{\mathcal{T}_{x_t}\M}\left(c_t\right) = 0$, then $v_t = -\P_{\mathcal{T}_{x_t}\M}\left(x_t\right)$; otherwise, there exists a constant $\theta_t > 0$ such that $-\theta_t \cdot \P_{\mathcal{T}_{x_t}\M}\left(c_t\right) = \P_{\mathcal{T}_{x_t}\M}\left(x_t\right)+ v_t$, i.e., $v_t = -\theta_t \cdot \P_{\mathcal{T}_{x_t}\M}\left(c_t\right) - \P_{\mathcal{T}_{x_t}\M}\left(x_t\right)$. Hence, we obtain $\theta_t\left(\|x_t -\theta_t \cdot \P_{\mathcal{T}_{x_t}\M}\left(c_t\right) - \P_{\mathcal{T}_{x_t}\M}\left(x_t\right)\|^2 + 1\right) = 1$. It follows that $\theta_t$ satisfies    \begin{equation}\label{eq.one dimensional equation}
      \|\P_{\mathcal{T}_{x_t}\M}\left(c_t\right)\|^2 \theta^3 + \left(\|\P_{\mathcal{N}_{x_t}\M}\left(x_t\right)\|^2 + 1\right)\theta - 1 = 0.  
    \end{equation}
    Clearly, $\theta_t$ is the unique positive solution of above equation. By Cardano's formula, $\theta_t$ can be expressed in a closed form.
\end{proof}

\begin{proposition}\label{propo.sphere closed form}
    Suppose the nonsmooth term $g \equiv 0$, and consider the sphere $\M = \mathbb{S}^{n-1}$. If the reference function is chosen as $h(x) = -\sum_{i=1}^n \log x_{i}$ or $h(x) = \sum_{i=1}^n x_{i}\log x_{i}$, then solving the subproblem \eqref{eq.update nabla} reduces to solving a one-dimensional nonlinear equation.
\end{proposition}
\begin{proof}
Recall that the tangent space of the sphere at $x_t$ is $\T_{x_t}\mathbb S^{\,n-1}=\{v\in\reals^{n}: x_t^{\top}v=0\}$. Hence the subproblem~\eqref{eq.simplified subproblem} becomes
\begin{align*}
    \min_{v\in\reals^{n}} \quad &\inner{c_t}{v} + h(x_t+v) \\
    \operatorname{s.t.} \quad &x_t^\top v = 0.
\end{align*}
By associating the Lagrange multiplier $\lambda\in\reals$, the Lagrangian function is $\mathcal L(v,\lambda) =\inner{c_t}{v}+h(x_t+v)+\lambda\,x_t^{\top}v$. The KKT conditions are
\[
c_t+\nabla h(x_t+v)+\lambda x_t=0, \quad x_t^{\top}v=0.
\]
For the log-barrier reference function $h(x)= -\sum_{i=1}^n \log x_{i}$, the first KKT condition gives
\(
c_t+\lambda x_t=(x_t+v)^{\odot-1},
\)
where the notation ``$\odot-1$" denotes the element-wise inverse. Thus, $v=(c_t+\lambda x_t)^{\odot-1}-x_t$. Substituting into $x_t^{\top}v=0$ yields the scalar equation
\[
\sum_{i=1}^{n}\frac{x_{t,i}}{c_{t,i}+\lambda x_{t,i}}-1=0,
\]
which is strictly decreasing and therefore has a unique root $\lambda^*$. With this root, the subproblem solution is \(v_t=(c_t+\lambda^* x_t)^{\odot-1}-x_t\). Similarly, for the entropy reference function $h(x) = \sum_{i=1}^n x_{i}\log x_{i}$, the first KKT relation becomes
\(c_t+\log(x_t+v)+\mathbf1_n+\lambda x_t=0\),
where $\exp(\cdot)$ is element-wise exponential, and $\mathbf1_n$ is the all-ones vector in~$\reals^{n}$. Plugging this into $x_t^{\top}v=0$ gives
\[
\sum_{i=1}^{n}x_{t,i}\exp\left(-c_{t,i}-\lambda x_{t,i}-1\right)-1=0.
\]
Again, it is a strictly decreasing function with a unique root $\lambda^*$.  The corresponding direction is
\(v_t=\exp\left(-c_t-\lambda^* x_t-\mathbf1_n\right)-x_t\).
\end{proof}

\section{Projection-Based Riemannian Bregman Gradient Method}\label{section.projection}
As an alternative to the retraction-based approach described above, the classical projection method can also solve problem~\eqref{eq.main} efficiently \citep{hu2024projected,zhang2024nonconvex,ding2024convergence}. For certain special submanifolds, it is possible to directly project onto the manifold; that is, one can easily compute $\P_\M(x + v)$, where $x \in \M$ and $v \in \reals^n$. It is well known that for the Stiefel manifold, the projection can be computed via polar decomposition, and similar easily computable projections exist for the Grassmannian and fixed-rank manifold cases \citep{absil2012projection,ding2024convergence}. The main advantage of the projection-based approach is that the update direction can be computed in the full ambient Euclidean space without being restricted to the tangent space. This simplification removes the tangent-space constraint, thus avoiding constrained subproblems such as \eqref{eq.update nabla} in Algorithm~\ref{alg.relative GD}. Consequently, at each iteration, determining the update direction reduces to solving an Euclidean unconstrained optimization problem.

In this section, we develop an efficient projection-based Bregman gradient method for smooth Riemannian optimization problems (\( g \equiv 0 \) in problem \eqref{eq.main}). At iteration $t$, we solve the following unconstrained subproblem:
\begin{equation}\label{eq.projection subproblem}
            v_t \ = \ \argmin_{v \in \reals^n} \ \inner{\grad f(x_t)}{v} + \gamma_t D_h(x_t + v, x_t).
        \end{equation}
By the first‑order optimality condition of \eqref{eq.projection subproblem}, $v_t = 0$ if and only if $\grad f(x_t) = 0$. Hence $\|v_t\|$ also serves as a valid stationarity measure in the projection-based framework, analogous to the retraction-based case.
\begin{definition}\label{eq.stationary point projection}
   Given accuracy $\epsilon > 0$, we say $x_t$ is an $\epsilon$-approximate Riemannian stationary point of problem \eqref{eq.main} with $g \equiv 0$ whenever $\|v_t\| \le \epsilon$, where $v_t$ is defined in \eqref{eq.projection subproblem}. 
\end{definition}
Since the tangent-space constraint is removed, we decompose $v_t$ into its tangential and normal components: $v_t = v_t^\T + v_t^\N$, where $v_t^\T = \P_{\T_{x_t}\M}(v_t)$ and $v_t^\N = \P_{\N_{x_t}\M}(v_t)$. As $v_t$ may have a large component in the normal space, we introduce a correction normal vector \(u_t \in \N_{x_t}\M\) in the update step. This correction prevents the projection from introducing large deviations due to the normal component. Because $u_t$ is chosen after computing $v_t$, we control its size via $\|u_t\| \le \tau \|v_t\|$ for some parameter $\tau \ge 0$. For example, one may simply select $u_t = -v_t^\N$ at every iteration, and then $\tau = 1$. Finally, we summarize the projection-based Riemannian Bregman gradient method in Algorithm~\ref{alg.projection relative GD}.
\begin{algorithm}[h]
	\caption{Projection-Based Riemannian Bregman Gradient Method}
	\label{alg.projection relative GD}
	\begin{algorithmic}[1] 
        \item \textbf{Input:} initial point $x_0 \in \M$, $\gamma_t \ge L$, $\tau \ge 0$, $\rho \in (0,1)$
        \item \textbf{For} $t = 0, 1, \ldots$ \textbf{do}
        \item \quad \quad Obtain $v_t$ by solving the subproblem \eqref{eq.projection subproblem}
        \item \quad \quad Choose the correction normal vector $u_t \in \N_{x_t}\M$ satisfying $\|u_t\| \le \tau \|v_t\|$
         \item \quad \quad Set the initial stepsize $\alpha_t = 1$
         \item  \quad \quad \textbf{While} $F\left(\P_{\M}(x_t + \alpha_t(v_t + u_t))\right) - F(x_t) \ > \ -\frac{\gamma_t\lambda\alpha_t}{4} \|v_t\|^2$ \textbf{do}
          \item  \quad \quad \quad \quad $\alpha_t := \rho \alpha_t$
          \item  \quad \quad \textbf{ end While}
           \item \quad \quad Update $x_{t+1} = \P_{\M}(x_t + \alpha_t(v_t + u_t))$
	\end{algorithmic}
\end{algorithm}

Before delving into the theoretical analysis of Algorithm~\ref{alg.projection relative GD}, we first provide some properties of the projection operator onto a differentiable submanifold. The following result comes from Lemma 4 in \cite{absil2012projection} and Lemma 5.2 in \cite{ding2024convergence}.
\begin{lemma}\label{lemma.projection radius}
    Let $\M \subseteq \mathbb{R}^n$ be a submanifold of class $C^k$ with $k \geq 2$. Given any $x \in \M$, there exists $\varrho(x)>0$ such that $\mathcal{P}_{\M}(y)$ uniquely exists for all $y \in\mathbb{B}\left(x, \varrho(x)\right)$. Moreover, $\mathcal{P}_{\M}(y)$ is of class $C^{k-1}$ for $y \in\mathbb{B}\left(x, \varrho(x)\right)$, and its differential at $x$ satisfies $\mathrm{D} \mathcal{P}_{\M}(x) = \mathcal{P}_{\T_{x} \M}$. Additionally, for any $y \in \M \cap\mathbb{B}\left(x, \varrho(x)\right)$ and $w \in \N_y\M$ satisfying $y + w \in\mathbb{B}\left(x, \varrho(x)\right)$, we have $\P_\M(y + w) = y$.
\end{lemma}


However, the above properties hold only locally, as the projection radius depends on the point $x$. Hence, the projection radius sequence $\{\varrho(x_t)\}_{t \ge 0}$ may converge to zero, where the iterates $x_t$, $t \ge 0$ are generated by Algorithm~\ref{alg.projection relative GD}. To prevent such pathological behavior, we choose the initial point $x_0 = \widetilde{x}$. Therefore, as we will demonstrate later, the iterates generated by Algorithm~\ref{alg.projection relative GD} remain within $\mathcal{L}(\widetilde{x})$ due to the backtracking linesearch. Moreover, for any $x \in \mathcal{L}(\widetilde{x})$, we can derive the following projection inequalities.

\begin{lemma}\label{lemma.projection inequality}
    Suppose Assumption~\ref{assumption} holds. For any $x \in \mathcal{L}(\widetilde{x})$, there exists a constant $\varrho > 0$ such that for any vectors $v \in \T_x\M$, $u \in \N_x\M$ satisfying $\|v + u\| \leq \varrho/2$, we have
    $$
        \begin{aligned}
        \left\|\mathcal{P}_{\M}(x+v+u)-x\right\| & \ \leq \ M_{1}^\P \|v\|, \\
        \left\|\mathcal{P}_{\M}(x+v+u)-x-v\right\| & \ \leq \  M_{2}^\P \|v\|^2+M_{3}^\P\|v\|\|u\|,
        \end{aligned}
    $$
    for some positive constants $M_{1}^\P, M_{2}^\P, M_{3}^\P > 0$.
\end{lemma}
\begin{proof}
   First notice that $\mathcal{L}(\widetilde{x}) \subseteq \cup_{z \in \mathcal{L}(\widetilde{x})} \mathbb{B}(z, \varrho(z) / 2)$, where $\varrho(z)$ is given in Lemma \ref{lemma.projection radius}. Because \(\mathcal{L}(\widetilde{x})\) is compact, by Lebesgue covering lemma, the open cover admits a finite sub‑cover: there exist finite points \(z_1,\dots,z_m\in \mathcal{L}(\widetilde{x})\) such
    that $\mathcal{L}(\widetilde{x}) \subseteq \cup_{i=1}^{m} \mathbb B \left(z_i,\varrho(z_i) / 2\right)$. Define $\varrho \triangleq \min_{1\le i\le m}\varrho(z_i) / 2 > 0$. Then given $x \in \mathcal{L}(\widetilde{x})$, there exists a sub-cover such that $x \in \mathbb{B}\left(z_i,\varrho(z_i) / 2\right)$ for some $z_i$. Hence, for any $y \in \mathbb{B}(x,\varrho)$, $\|y - z_i\| \le \|y - x\| + \|x - z_i\| \le \varrho + \varrho(z_i) / 2 \le \varrho(z_i)$, which says $y \in \mathbb{B}(z_i, \varrho(z_i))$. Consequently, by Lemma \ref{lemma.projection radius}, $\mathcal{P}_{\M}(y)$ uniquely exists and is of class $C^\infty$. Its differential at $x$ satisfies $\mathrm{D} \mathcal{P}_{\M}(x) = \mathcal{P}_{\T_{x} \M}$. Then both $\P_\M(\cdot)$ and $\mathrm{D} \mathcal{P}_{\M}(\cdot)$ are Lipschitz continuous on $\mathbb{B}(x, \varrho)$, with Lipschitz constants $L_{\P_\M}(x)$ and $L_{\mathrm{D}\P_\M}(x)$, respectively.  Besides, for any $w \in \N_x\M$ satisfying $\|w\| \le \varrho$, we also have $\|x + w - z_i\| \le \|w\| + \|x - z_i\| \le \varrho(z_i)$. It follows $\P_\M(x + w) = x$. 
    
   Now we prove two inequalities. Since $\|v + u\| \le \varrho / 2$, then $x + v + u \in \mathbb{B}(x, \varrho)$ and $\P_\M(x + u) = x$. It holds that
   \begin{align*}
     \left\|\mathcal{P}_{\M}(x+v+u) - x\right\| \ = \ \left\|\mathcal{P}_{\M}(x+v+u) - \P_\M(x + u)\right\| \ \le \ L_{\P_\M}(x) \|v\| \ \le \ M_{1}^\P \|v\|,
   \end{align*}
   which gives the first inequality with \(M_{1}^\P \triangleq \max_{x \in \mathcal{L}(\widetilde{x})}L_{\mathcal{P}_{\M}}(x)\). As for the second inequality, consider the first-order Taylor expansion of \(\mathcal{P}_{\M}\) at \(x + u\). It follows
   \begin{align*}
       \left\|\mathcal{P}_{\M}(x+u +v) - \P_\M(x + u) - \mathrm{D} \mathcal{P}_{\M}(x + u)[v] \right\| \ \le \ \frac{L_{\mathrm{D}\P_\M}(x)}{2}\|v\|^2.
   \end{align*}
    Besides, since $\mathrm{D} \mathcal{P}_{\M}(x) = \mathcal{P}_{\T_{x} \M}$, we have $\mathrm{D} \mathcal{P}_{\M}(x)[v] = \mathcal{P}_{\T_{x} \M}(v) = v$. Using the Lipschitz continuity of \(\mathrm{D} \mathcal{P}_{\M}\), it yields
    \begin{align*}
        \|\mathrm{D} \mathcal{P}_{\M}(x + u)[v] - v\| \ = \ \|\mathrm{D} \mathcal{P}_{\M}(x + u)[v] - \mathrm{D} \mathcal{P}_{\M}(x)[v]\| \ \le L_{\mathrm{D}\P_\M}(x) \|v\|\|u\|.
        \end{align*}
   Recalling that \(\mathcal{P}_{\M}(x + u) = x\). Let $M_{2}^\P \triangleq \max_{x \in \mathcal{L}(\widetilde{x})}L_{\mathrm{D}\P_\M}(x) / 2$, and $M_{3}^\P \triangleq 2M_{2}^\P$. By combining the above two equations, we conclude
    \begin{align*}
       \left\|\mathcal{P}_{\M}(x+u +v) - x - v \right\| \ \le \ M_{2}^\P\|v\|^2 + M_{3}^\P \|v\|\|u\|.
   \end{align*}
   The proof is completed.
\end{proof}
\begin{remark}\label{remark.compare with ding}
    The above result follows from Lemma 5.10 in \cite{ding2024convergence} (We restate it in the Appendix), which was originally proved for compact submanifolds. On a compact submanifold, one can always guarantee that the projection $\P_\M(x + v + u)$ remains on the manifold, so only a bound on the normal component is needed. In our setting, however, we work on a compact subset $\mathcal{L}(\widetilde{x})$ of a (potentially non-compact) differentiable submanifold, and $\P_\M(x + v + u)$ may not remain in $\mathcal{L}(\widetilde{x})$. Therefore, in order to prove the projection inequalities, we must control the magnitude of both tangent and normal vectors. 
\end{remark}

Since the update direction in Algorithm~\ref{alg.projection relative GD} generally contains components in the normal space, it is necessary to control the magnitude of its normal component. The following lemma proves that $\|\P_{\N_x\M}(x - y)\| = \mathcal{O}(\|x-y\|^2)$. 

\begin{lemma}\label{lemma.bound projection normal space}
    Suppose Assumption~\ref{assumption} holds. For any $x \in \mathcal{L}(\widetilde{x})$, we have $\|\P_{\N_x\M}(x - y)\| \le M_4^\P\|x - y\|^2$ for some $M_4^\P > 0$.
\end{lemma}
\begin{proof}
    We first argue that there exists a constant $\chi > 0$ such that for any $y \in \overline{\mathbb{B}}(x, \chi) \cap \M$, $\|\P_{\N_x\M}(x - y)\| = \mathcal{O}(\|x-y\|^2)$ holds. By Definition \ref{def.submanifold} and Corollary \ref{coro.tangent space charactirization}, for any $z \in \mathcal{L}(\widetilde{x})$, there exists an open neighborhood $\mathcal{U}_z$ of $z$ and a $C^k$ map $\phi_z: \mathcal{U}_z \to \reals^{n-d}$ such that $\mathcal{U}_z \cap \mathcal{M} = \{y\in\mathcal{U}_z : \phi_z(y)=0\}$. Besides, for any $y\in \mathcal{U}_z\cap\mathcal{M}$, we have $\operatorname{rank}(\operatorname{J}\phi_z(y))=n-d$, and $\T_y\M = \operatorname{ker}(\operatorname{J}\phi_z(y))$. Since $\mathcal{U}_z$ is open, we can choose a ball centered at $z$ with radius $\chi(z) > 0$ such that $\overline{\mathbb{B}}(z, \chi(z)) \subseteq \mathcal{U}_z$. Clearly, $\cup_{z \in  \mathcal{L}(\widetilde{x})} \mathbb{B}(z, \chi(z)/2)$ forms an open cover of $\mathcal{L}(\widetilde{x})$. Since $\mathcal{L}(\widetilde{x})$ is compact, then there exists a finite sub-cover such that $\mathcal{L}(\widetilde{x}) \subseteq \cup_{i=1}^m \mathbb{B}(z_i, \chi(z_i)/2)$. 
    
    Now we choose $\chi \triangleq \min_{i=1, \ldots, m} \chi(z_i)/2$. Hence for any $x \in \mathcal{L}(\widetilde{x})$ and $y \in \overline{\mathbb{B}}(x, \chi) \cap \M$, there exists some $i$ such that $x,y \in \overline{\mathbb{B}}(z_i, \chi(z_i)) \subseteq \mathcal{U}_{z_i}$. Consequently, $\phi_{z_i}(y) = \phi_{z_i}(x) = 0$, and $\T_x\M = \operatorname{ker}(\operatorname{J} \phi_{z_i}(x))$. Then it holds that $ \P_{\N_{x}\M}(y - x) = \operatorname{J}\phi_{z_i}(x)^\top (\operatorname{J}\phi_{z_i}(x) \operatorname{J}\phi_{z_i}(x)^\top)^{-1}\operatorname{J}\phi_{z_i}(x) (y - x)$. Let $L_{\operatorname{J}\phi_{z_i}}$ be the Lipschitz constant of $\operatorname{J}\phi_{z_i}(\cdot)$ on the closed ball $\overline{\mathbb{B}}(z_i, \chi(z_i))$. We obtain
    \begin{align*}
        \|\P_{\N_{x}\M}(y - x)\| \ 
        = \ &\|\operatorname{J}\phi_{z_i}(x)^\top (\operatorname{J}\phi_{z_i}(x) \operatorname{J}\phi_{z_i}(x)^\top)^{-1}\operatorname{J}\phi_{z_i}(x) (y - x)\| \\
        \le \ &\|\operatorname{J}\phi_{z_i}(x)^\top (\operatorname{J}\phi_{z_i}(x) \operatorname{J}\phi_{z_i}(x)^\top)^{-1}\| \cdot \|\operatorname{J}\phi_{z_i}(x) (y-x)\| \\
        = \ &\|\operatorname{J}\phi_{z_i}(x)^\top (\operatorname{J}\phi_{z_i}(x) \operatorname{J}\phi_{z_i}(x)^\top)^{-1}\| \cdot \|\phi_{z_i}(y) - \phi_{z_i}(x) - \operatorname{J}\phi_{z_i}(x) (y - x)\| \\
        \le \ &\|\operatorname{J}\phi_{z_i}(x)^\top (\operatorname{J}\phi_{z_i}(x) \operatorname{J}\phi_{z_i}(x)^\top)^{-1}\| \cdot \frac{L_{\operatorname{J}\phi_{z_i}}}{2}\|y - x\|^2.
    \end{align*}
    By choosing $M_{\operatorname{J} \phi} \triangleq \max_{i = 1, \ldots, m}  L_{\operatorname{J}\phi_{z_i}} \max_{x \in \overline{\mathbb{B}}(z_i, \chi(z_i))}\|\operatorname{J}\phi_{z_i}(x)^\top (\operatorname{J}\phi_{z_i}(x) \operatorname{J}\phi_{z_i}(x)^\top)^{-1}\|$, the inequality $\|\P_{\N_x\M}(x - y)\| = \mathcal{O}(\|x-y\|^2)$ holds. For those $y \in \M$ satisfying $\|y - x\| > \chi$, it follows
    \begin{align*}
       \|\P_{\N_x\M}(y - x)\| \ \le \ \|x - y\| \ \le \ \frac{\|x - y\|^2}{\chi}.
    \end{align*}
    Set $M_4^\P = \max\{M_{\operatorname{J} \phi}, 1 / \chi\}$. The proof is completed.  
\end{proof}

Compared with Algorithm~\ref{alg.relative GD}, here we need an additional assumption to control growth of the gradient of the reference function $h$ due to the existence of normal vectors. 
\begin{assumption}\label{assumption.correction u}
    The reference function $h$ is twice continuously differentiable.
\end{assumption}

The above requirement is naturally satisfied by many widely-used reference functions, such as $h(x) = \frac{1}{4}\|x\|^4 + \frac{1}{2}\|x\|^2$ and $h(x) = -\sum_{i=1}^n \log(x_i)$. Now we move to the theoretical analysis of Algorithm~\ref{alg.projection relative GD}. In the following analysis, we assume the update direction $v_t \neq 0$ at iteration $t$; otherwise, $x_t$ is already a stationary point. Suppose the current iterate $x_t \in \mathcal{L}(\widetilde{x})$. By Lemma \ref{lemma.projection inequality}, if we choose the stepsize such that \(\alpha_t \le \min\{1, \varrho / (2\|v_t + u_t\|)\}\), we have $\|x_{t+1} - x_t\| = \|\P_{\M}(x_t + \alpha_t(v_t + u_t)) - x_t\| \le \alpha_t M_{1}^\P\|v_t^\T\| \le M_{1}^\P\varrho$. Therefore, we define the maximum gradient norm of the reference function over the ball of radius $M_{1}^\P\varrho$ around $x_t$ as $G_h(x_t, M_{1}^\P\varrho) \triangleq \max_{x \in \overline{\mathbb{B}}(x_t, M_{1}^\P\varrho)} \|\nabla h(x)\|$. Correspondingly, for analysis purposes, we define the maximum Hessian norm of the reference function over the ball $\overline{\mathbb{B}}(x_t, M_{1}^\P\varrho)$ as $H_h(x_t, M_{1}^\P\varrho) \triangleq \max_{x \in \overline{\mathbb{B}}(x_t, M_{1}^\P\varrho)} \|\nabla^2 h(x)\|$. Then, under these assumptions, we can rigorously establish that the deviation between the Bregman distances evaluated at \(v_t\) and \(v_t + u_t\) in the update step is bounded.

\begin{lemma}\label{lemma.projection Bregman deviation}
    Suppose Assumptions \ref{assumption} and \ref{assumption.correction u} hold. Fix an iterate $x_t \in \mathcal{L}(\widetilde{x})$. For any \(\alpha_t \le \min\{1, \varrho / (2\|v_t + u_t\|)\}\), it holds that 
    \begin{align*}
       D_h(x_{t+1}, x_t) - D_h(x_t + \alpha_t v_t, x_t) \
       \le \ \Psi_1(x_t) \|\alpha_t v_t\|^2, 
    \end{align*}
    where $\Psi_1(x_t) \triangleq 2G_h(x_t, M_{1}^\P\varrho) (M_{2}^\P + M_{3}^\P)(1 + \tau) + H_h(x_t, M_{1}^\P\varrho)$.
\end{lemma}
\begin{proof}
  First notice that
    \begin{align*}
        &D_h(x_{t+1}, x_t) - D_h(x_t + \alpha_t v_t, x_t) \\
        = \ &D_h(x_{t+1}, x_t) - D_h(x_t + \alpha_t v_t^\T, x_t) + D_h(x_t + \alpha_t v_t^\T, x_t) - D_h(x_t + \alpha_t v_t, x_t),
    \end{align*}
    where we use the Bregman divergence between $x_t$ and $x_t + \alpha_t v_t^\T$ as an intermediate term. On the one hand, due to the convexity of $h$, we have
    \begin{align*}
      &D_h(x_{t+1}, x_t) - D_h(x_t + \alpha_t v_t^\T, x_t) \\
      = \ &h(x_{t+1}) - h(x_t + \alpha_t v_t^\T) - \inner{\nabla h(x_t)}{x_{t+1} - x_t - \alpha_t v_t^\T} \\
      \le \ &\inner{\nabla h(x_{t+1}) - \nabla h(x_t)}{x_{t+1} - x_t - \alpha_t v_t^\T} \\
      \le \ &\|\nabla h(x_{t+1}) - \nabla h(x_t)\| \cdot \|\P_{\M}(x_t + \alpha_t (v_t + u_t)) - x_t - \alpha_t v_t^\T\| \\
      \le \ &2G_h(x_t, M_{1}^\P\varrho) \|\P_{\M}(x_t + \alpha_t (v_t + u_t)) - x_t - \alpha_t v_t^\T\|. 
    \end{align*}
    Since the stepsize $\alpha_t$ satisfies $\alpha_t \le \varrho / (2\|v_t + u_t\|)$, using the second inequality in Lemma~\ref{lemma.projection inequality} implies
    \begin{align*}
       \|\P_{\M}(x_t + \alpha_t (v_t + u_t)) - x_t - \alpha_t v_t^\T\| \ \le \ M_{2}^\P\|\alpha_t v_t^\T\|^2 + M_{3}^\P\|\alpha_t v_t^\T\|\|\alpha_t (v_t^\N + u_t)\|.
    \end{align*}
    Note that $\|v_t^\T\| \le \|v_t\|$, $\|v_t^\N + u_t\| \le \|v_t\| + \|u_t\|$, and $\|u_t\| \le \tau\|v_t\|$. Hence, we obtain $D_h(x_{t+1}, x_t) - D_h(x_t + \alpha_t v_t^\T, x_t) \le 2 G_h(x_t, M_{1}^\P\varrho) (M_{2}^\P + M_{3}^\P)(1 + \tau) \|\alpha_t v_t\|^2$.
    On the other hand, we also have
    \begin{align*}
       &D_h(x_t + \alpha_t v_t^\T, x_t) - D_h(x_t + \alpha_t v_t, x_t) \\
      = \ &h(x_t + \alpha_t v_t^\T) - h(x_t + \alpha_t v_t) - \alpha_t\inner{\nabla h(x_t)}{v_t^\T - v_t} \\
      \le \ &\alpha_t\inner{\nabla h(x_t + \alpha_t v_t^\T) - \nabla h(x_t)}{-v_t^\N}  \\
      \le \ &\alpha_t\|\nabla h(x_t + \alpha_t v_t^\T) - \nabla h(x_t)\| \cdot \|v_t\|.
    \end{align*}
    Since $h$ is twice continuously differentiable, Newton-Leibniz formula yields that
    \begin{align*}
      \|\nabla h(x_t + \alpha_t v_t^\T) - \nabla h(x_t)\| \ = \ \|\int_0^1 \nabla^2 h(x_t + \alpha_t v_t^\T \cdot t) \alpha_t v_t^\T \d t\| \ \le \ H_h(x_t, M_{1}^\P\varrho) \cdot \alpha_t \|v_t^\T\|.
    \end{align*}
    Thus we obtain $ D_h(x_t + \alpha_t v_t^\T, x_t) - D_h(x_{t}^+, x_t)) \le H_h(x_t, M_{1}^\P\varrho) \cdot \|\alpha_t v_t\|^2$. Combining the above inequalities yields 
    \begin{align*}
        &D_h(x_{t+1}, x_t) - D_h(x_t + \alpha_t v_t, x_t) \\
        \le \ &\left(2G_h(x_t, M_{1}^\P\varrho) (M_{2}^\P + M_{3}^\P)(1 + \tau) + H_h(x_t, M_{1}^\P\varrho)\right)\|\alpha_t v_t\|^2,  
    \end{align*}
    which completes the proof.
\end{proof}

\begin{lemma}\label{lemma.descent property projection}
    Suppose Assumptions \ref{assumption} and \ref{assumption.correction u} hold. Fix an iterate $x_t \in \mathcal{L}(\widetilde{x})$. For any $\gamma_t \ge L$ and $\alpha_t \in (0,1)$, there exists a constant $\alpha_t^\prime > 0$ such that for any $0 < \alpha_t \le \alpha_t^\prime$, the next iterate $x_{t+1}$ in Algorithm \ref{alg.projection relative GD} satisfies
     \begin{align*}
        F(x_{t+1}) - F(x_t) \ \le \ -\frac{\gamma_t \lambda \alpha_t}{4}\|v_t\|^2,
    \end{align*} 
    where $\alpha_t^\prime$ is defined in \eqref{eq.alpha bar projection}.
\end{lemma}
\begin{proof}
   By relative smoothness of $f$, it follows
    \begin{align*}
        f(x_{t+1}) - f(x_t) \ \le \ &\inner{\nabla f(x_t)}{x_{t+1} - x_t} + \gamma_t D_h(x_t + \alpha_t v_t, x_t) \\
        &+ \gamma_t D_h(x_{t+1}, x_t) - \gamma_t D_h(x_t + \alpha_t v_t, x_t).
    \end{align*}
    Choosing the stepsize \(\alpha_t\) such that \(\alpha_t \le \varrho / (2\|v_t + u_t\|)\), and applying Lemma~\ref{lemma.projection Bregman deviation}, it yields
    \begin{align*}
        f(x_{t+1}) - f(x_t) \ \le \ \inner{\nabla f(x_t)}{x_{t+1} - x_t} + \gamma_t D_h(x_t + \alpha_t v_t, x_t) + \gamma_t\Psi_1(x_t)\|\alpha_t v_t\|^2.
    \end{align*}
    The inner product term can be upper bounded as
    \begin{align*}
        &\inner{\nabla f(x_t)}{x_{t+1} - x_t} \\
        = \ &\inner{\grad f(x_t)}{x_{t+1} - x_t} + \inner{\P_{\N_{x_t}\M}(\nabla f(x_t))}{x_{t+1} - x_t} \\
        = \ &\inner{\grad f(x_t)}{x_{t+1} - x_t} + \inner{\P_{\N_{x_t}\M}(\nabla f(x_t))}{\P_{\N_{x_t}\M}(x_{t+1} - x_t)} \\
        \le \ &\inner{\grad f(x_t)}{x_{t+1} - x_t} + \|\P_{\N_{x_t}\M}(\nabla f(x_t))\| \cdot \|\P_{\N_{x_t}\M}(x_{t+1} - x_t)\| \\
        \le \ &\inner{\grad f(x_t)}{x_{t+1} - x_t} + \|\nabla f(x_t)\| \cdot \|\P_{\N_{x_t}\M}(x_{t+1} - x_t)\|.
    \end{align*}
    We now estimate $\|\P_{\N_{x_t}\M}(x_{t+1} - x_t)\|$. Recall that $\|x_{t+1} - x_t\| = \|\P_{\M}(x_t + \alpha_t(v_t + u_t)) - x_t\| \le \alpha_t M_{1}^\P\|v_t^\T\|$ when $\alpha_t \le \varrho / (2\|v_t + u_t\|)$. By using Lemma \ref{lemma.bound projection normal space}, $\|\P_{\N_{x_t}\M}(x_{t+1} - x_t)\| \le M_4^\P\|x_{t+1} - x_t\|^2 \le (M_{1}^\P)^2 M_4^\P \|\alpha_t v_t\|^2$. Let $x_t^+ \triangleq x_t + \alpha_t (v_t + u_t)$. Recall that the correction normal vector $u_t \in \N_{x_t}\M$. Hence, we obtain
    \begin{align*}
       &\inner{\grad f(x_t)}{x_{t+1} - x_t} \\
       = \ &\inner{\grad f(x_t)}{x_{t+1} - x_t^+} + \inner{\grad f(x_t)}{x_t^+ - x_t} \\
       = \ &\inner{\grad f(x_t)}{x_{t+1} - x_t - \alpha_t v_t^\T} + \alpha_t \inner{\grad f(x_t)}{v_t} \\
       \le \ &\|\grad f(x_t)\| \cdot \|\P_{\M}(x_t + \alpha_t (v_t + u_t)) - x_t - \alpha_t v_t^\T\| + \alpha_t \inner{\grad f(x_t)}{v_t} \\
       \le \ &\|\nabla f(x_t)\| \cdot \left((M_{2}^\P + M_{3}^\P)(1 + \tau) \|\alpha_t v_t\|^2\right) + \alpha_t \inner{\grad f(x_t)}{v_t},
    \end{align*}
    where we use the same argument in the proof of Lemma \ref{lemma.projection Bregman deviation}. Let $\Psi_2(x_t) \triangleq (1 + \tau)(M_{2}^\P + M_{3}^\P)\|\nabla f(x_t)\| + (M_{1}^\P)^2M_4^\P\|\nabla f(x_t)\|$. By Combining the above inequalities, we have
     \begin{align*}
        f(x_{t+1}) - f(x_t) \ \le \
    \left(\gamma_t\Psi_1(x_t) + \Psi_2(x_t)\right)\|\alpha_t v_t\|^2 + \alpha_t \inner{\grad f(x_t)}{v_t} + \gamma_t D_h(x_t + \alpha_t v_t, x_t).
    \end{align*}
    To proceed, we use the first-order optimality condition of subproblem~\eqref{eq.projection subproblem}, which implies $ \grad f(x_t) + \gamma_t \nabla h(x_t + v_t) - \gamma_t \nabla h(x_t) = 0$. From Lemma~\ref{lemma.bound bregman deviation}, we also have $D_h(x_t + \alpha_t v_t, x_t) \le \alpha_t  D_h(x_t + v_t, x_t)$ for any $\alpha_t \in (0,1)$. We now derive
    \begin{align*}
       &\alpha_t\inner{\grad f(x_t)}{v_t} + \gamma_t D_h(x_t + \alpha_t v_t, x_t) \\
       = \ & \alpha_t \gamma_t \inner{ \nabla h(x_t) - \nabla h(x_t + v_t )}{v_t} + \alpha_t \gamma_t D_h(x_{t} + v_t, x_t) \\
       = \ & \alpha_t\gamma_t \left(h(x_t + v_t) - h(x_t) - \inner{\nabla h(x_t + v_t)}{v_t} \right) \\
       \le \ &-\frac{\alpha_t \gamma_t \lambda}{2}\|v_t\|^2,
    \end{align*}
    where the last inequality follows from the $\lambda$-strong convexity of $h$. Therefore, we obtain
     \begin{align*}
        f(x_{t+1}) - f(x_t) \ \le \ &\left(\gamma_t\Psi_1(x_t) + \Psi_2(x_t)\right) \|\alpha_t v_t\|^2 -\frac{\alpha_t \gamma_t \lambda}{2}\|v_t\|^2.
    \end{align*}
    Let
    \begin{equation}\label{eq.alpha bar projection}
       \alpha_t^\prime = \min\left\{\frac{\varrho}{2\|v_t + u_t\|}, \frac{\gamma_t\lambda}{4 \left(\gamma_t\Psi_1(x_t) + \Psi_2(x_t)\right)}\right\}. 
    \end{equation}
    Then, since $g \equiv 0$, we conclude that for any $0 < \alpha_t \le \min\{1, \alpha_t^\prime\}$,
   \begin{align*}
       F(x_{t+1}) - F(x_t) \ \le \ -\frac{\gamma_t\lambda\alpha_t}{4} \|v_t\|^2.
   \end{align*}
   The proof is completed.
\end{proof}

Similar to the retraction-based approach, we establish a per-iteration descent lemma for Algorithm~\ref{alg.projection relative GD}, which ensures that the backtracking line-search procedure is well-defined. Moreover, it follows by induction that the entire sequence of iterates \(\{x_t\}_{t\ge0}\) remains within \(\mathcal{L}(\widetilde{x})\) whenever \(x_0\in\mathcal{L}(\widetilde{x})\). Consequently, $\alpha_t^\prime$ in the above lemma admits a strictly positive lower bound, which further implies the following convergence result.

\begin{theorem}\label{thm.convergence of projection based}
   Suppose Assumptions \ref{assumption} and \ref{assumption.correction u} hold. Set the initial point $x_0 = \widetilde{x}$. Then every limit point of the sequence $\left\{x_t\right\}_{t \ge 0}$ generated by Algorithm \ref{alg.projection relative GD} with $\gamma_t = L$ satisfies the optimality condition of problem \eqref{eq.main}. Moreover, for any given accuracy \( \epsilon > 0 \), after at most \( \mathcal{O}(\epsilon^{-2}) \) iterations, Algorithm~\ref{alg.projection relative GD} with \( \gamma_t = L \) returns a direction \( v_t \) satisfying \( \|v_t\| \le \epsilon \).
\end{theorem}
\begin{proof}
    By Lemma \ref{lemma.continuity of v*}, we know that the minimizer of subproblem \ref{eq.projection subproblem} is continuous with respect to $x$. Since all iterates belong to $\mathcal{L}(\widetilde{x})$, which is compact, then the sequence \( \{v_t\}_{t \ge 0} \) is bounded by a uniform constant $\overline{v}$. Using an argument similar to the one used in the proof of Theorem~\ref{thm.convergence of retraction based}, it can be shown that the sequences \( \{\Psi_1(x_t)\}_{t \ge 0} \) and \( \{\Psi_2(x_t)\}_{t \ge 0} \) are bounded. Hence, there exist constants \(\Psi_1 > 0\) and \(\Psi_2 > 0\) such that \( \Psi_1(x_t) \le \Psi_1 \) and \( \Psi_2(x_t) \le \Psi_2 \) for all \( t \ge 0 \). Due to the backtracking linesearch, it holds that \( \alpha_t \ge \rho \alpha_t^\prime \). Then we can find a constant \( \alpha^\prime > 0 \) such that for all \( t \ge 0 \),
    \begin{align*}
        \alpha_t \ \ge \ &\min\left\{\frac{\rho \varrho}{2\|v_t + u_t\|}, \frac{\rho \gamma_t\lambda}{4 \left(\gamma_t\Psi_1(x_t) + \Psi_2(x_t)\right)}\right\} \\
        \ge \ & \min\left\{\frac{\rho \varrho}{2(1+\tau)\overline{v}}, \frac{\rho L \lambda}{4 \left(L \Psi_1 + \Psi_2\right)}\right\} \ \triangleq \ \alpha^\prime.
    \end{align*}
    By a similar argument in the proof of Theorem \ref{thm.convergence of retraction based}, we can conclude the sequence \( \{x_t\}_{t \ge 0} \) converges to a stationary point of problem \eqref{eq.main}, and \(\{x_t\}_{t \ge 0}\) admits at least one limit point. Besides, after at most $\mathcal{O}(\epsilon^{-2})$ iterations, the algorithm returns a direction \( v_t \) satisfying \( \|v_t\| \le \epsilon \).
\end{proof}

At the end of this section, we prove that for the fixed-rank manifold, the update direction generated by \eqref{eq.projection subproblem} coincides with the direction obtained in the retraction-based case when the reference function is chosen to be the quartic reference function.

\begin{proposition}\label{prop.projection fixed rank}
    Suppose the nonsmooth term $g \equiv 0$, and consider the fixed-rank manifold $\M_r = \{X \in \reals^{m \times p}: \operatorname{rank}(X) = r\}$, with $0 < r \le \min\{m,p\}$. If the reference function is chosen as $h(X) = \frac{1}{4}\|X\|^4 + \frac{1}{2}\|X\|^2$, then the direction $V_t$ generated by \eqref{eq.projection subproblem} is also the solution to \eqref{eq.update nabla}.
\end{proposition}
\begin{proof}
    For clarity, let $V_t^\P$ denote the solution to \eqref{eq.projection subproblem}, and let $V_t^{\operatorname{R}}$ denote the solution to \eqref{eq.update nabla}. Define $Y_t \triangleq X_t + V_t^\P$. Then $Y_t$ is the solution to the problem
    \begin{align*}
       Y_t \ = \ \argmin_{Y \in \reals^{m \times p}} \ \inner{\grad f(X_t)/\gamma_t}{Y} + D_h(Y, X_t). 
    \end{align*}
    Let $C_t^\prime \triangleq \grad f(X_t) / \gamma_t - \nabla h(X_t)$.
    By an argument similar to that in the proof of Proposition \ref{propo.quartic closed form}, we have $Y_t = -\theta_t C_t^\prime$, where $\theta_t$ is the unique positive solution to the equation $\|C_t^\prime\|^2\theta^3 + \theta - 1 = 0$. Thus, $V_t^\P = -\theta_t C_t^\prime - X_t$. Let $X_t \triangleq U_t \Sigma_t V_t^\top$ be the SVD decomposition. By Proposition 2.1 in \cite{vandereycken2012low}, the tangent space at $X_t$ is
    \begin{align*}
        \T_{X_t}\M_r = \{U_t M V_t^\top + UV_t^\top + U_t V^\top: M \in \reals^{r \times r}, U \in \reals^{m \times r}, V \in \reals^{p \times r}\}.
    \end{align*}
    Clearly, $X_t \in \T_{X_t}\M_r$ when $M = \Sigma_t$, $U = 0$, and $V = 0$. Hence, $\P_{\N_{X_t}\M_r}(X_t) = 0$. Moreover, for the chosen reference function $h(X) = \frac{1}{4}\|X\|^4 + \frac{1}{2}\|X\|^2$, we have $\nabla h(X_t) = (\|X_t\|^2 + 1)X_t$. Notice that $\nabla h(X_t) = \P_{\T_{X_t}\M_r}\left(\nabla h(X_t)\right)$, which implies that $C_t^\prime = \P_{\T_{X_t}\M_r}(C_t)$, where $C_t$ is defined in \eqref{eq.simplified subproblem}. Consequently, equation \eqref{eq.one dimensional equation} reduces to $\|C_t^\prime\|^2\theta^3 + \theta - 1 = 0$, and thus, $\theta_t$ is precisely the positive solution of \eqref{eq.one dimensional equation}. Therefore, by proposition \ref{propo.quartic closed form}, $V_t^{\operatorname{R}} = -\theta_t \cdot \P_{\mathcal{T}_{X_t}\M_r}\left(\nabla f(X_t) / \gamma_t - \nabla h(X_t)\right) - \P_{\mathcal{T}_{X_t}\M_r}\left(X_t\right) = -\theta_t C_t^\prime - X_t = V_t^\P$, completing the proof.
\end{proof}

\section{Extension to the Stochastic Setting}\label{section.stochastic}
In this section, we show that both retraction-based and projection-based Bregman gradient methods can be extended to the Riemannian stochastic optimization setting. Specifically, we consider the following optimization problem:
\begin{equation}\label{eq.main stochastic}
    \min _{x \in \mathcal{M}} F(x) = f(x) + g(x), \ \text{with} \ f(x) \triangleq \E_{\pi}[f(x,\pi)],
\end{equation}
where $\E_{\pi}$ is the expectation with respect to the random variable $\pi$. We assume access to a stochastic first-order oracle that returns gradients \(\nabla f(x, \pi)\), which are unbiased estimators of the true gradient with bounded variance. For all \(x \in \reals^n\), we have $\E_{\pi}\left[\nabla f(x,\pi)\right] = \nabla f(x)$, and $\E_{\pi}\left[\|\nabla f(x,\pi) - \nabla f(x)\|^2\right] \le \sigma^2$ with $\sigma > 0$.

For the retraction-based approach, at each iteration, we replace the full gradient in the update step \eqref{eq.update nabla} of Algorithm~\ref{alg.relative GD} with a stochastic estimator of the Euclidean gradient. We randomly sample a mini-batch \(\mathcal{B}_t\) and define $\nabla f_{\B_t}(x_t) \triangleq \frac{1}{|\B_t|}\sum_{j \in \B_t}\nabla f(x_t,\pi_t^{(j)})$, where \(\{\pi_t^{(j)}\}_{j \in \mathcal{B}_t}\) are i.i.d. samples drawn from the underlying distribution. The corresponding mini-batch Riemannian gradient follows $\grad f_{\B_t}(x_t) = \P_{\T_{x_t}\M}(\nabla f_{\B_t}(x_t))$. We then solve the subproblem~\eqref{eq.stochastic retraction update} using this stochastic gradient and update the iterate accordingly. The resulting procedure is summarized in Algorithm~\ref{alg.stochastic relative GD}. Similarly, for smooth Riemannian optimization problems, (\(g \equiv 0\) in problem \eqref{eq.main stochastic}), the projection-based Riemannian Bregman gradient method can be generalized in a similar manner. This stochastic variant is presented in Algorithm~\ref{alg.stochastic project GD}. In the stochastic case, we set $u_t = 0$ for simplicity.
\begin{algorithm}[h]
	\caption{Stochastic Retraction-Based Riemannian Bregman Gradient Method}
	\label{alg.stochastic relative GD}
	\begin{algorithmic}[1] 
        \item \textbf{Input:} initial point $x_0 \in \M$, $\gamma_t \ge L$, $\alpha_t > 0$
        \item \textbf{For} $t = 0, 1, \ldots$ \textbf{do}
        \item \quad \quad Obtain update direction $\zeta_t^{\operatorname{R}}$ by solving the subproblem
        \begin{equation}\label{eq.stochastic retraction update}
            \zeta_t^{\operatorname{R}} \ = \ \argmin_{\zeta \in \T_{x_t}\M} \ \inner{\nabla f_{\B_t}(x_t)}{\zeta} + \gamma_t D_h(x_t + \zeta, x_t) + g(x_t + \zeta)
        \end{equation}
        \item \quad \quad Update $x_{t+1} =  \retr(x_t,\alpha_t \zeta_t^{\operatorname{R}})$
	\end{algorithmic}
\end{algorithm}

\begin{algorithm}[h]
	\caption{Stochastic Projection-Based Riemannian Bregman Gradient Method}
	\label{alg.stochastic project GD}
	\begin{algorithmic}[1] 
        \item \textbf{Input:} initial point $x_0 \in \M$, $\gamma_t \ge L$, $\alpha_t > 0$
        \item \textbf{For} $t = 0, 1, \ldots$ \textbf{do}
        \item \quad \quad Obtain update direction $\zeta_t^{\P}$ by solving the subproblem
        \begin{equation}\label{eq.stochastic projection update}
            \zeta_t^{\P} \ = \ \argmin_{\zeta \in \reals^n} \ \inner{\grad f_{\B_t}(x_t)}{\zeta} + \gamma_t D_h(x_t + \zeta, x_t)
        \end{equation}
        \item \quad \quad Update $x_{t+1} = \P_{\M}(x_t + \alpha_t \zeta_t^{\P})$
	\end{algorithmic}
\end{algorithm}

To distinguish the update directions in the stochastic methods from those in the deterministic setting, we denote them by \(\zeta_t^{\operatorname{R}}\) and \(\zeta_t^{\P}\), corresponding to the retraction-based and projection-based methods, respectively. In the deterministic setting, we denote the update directions by \(v_t^{\operatorname{R}}\) and \(v_t^{\P}\). Recall that in Theorems~\ref{thm.convergence of retraction based} and~\ref{thm.convergence of projection based}, the norms \(\|v_t^{\operatorname{R}}\|\) and \(\|v_t^{\P}\|\) are used as measures of approximate stationarity, that is, $x_t$ is an $\epsilon$-approximate Riemannian stationary point if $\|v_t^{\operatorname{R}}\| \le \epsilon$ or $\|v_t^{\P}\| \le \epsilon$. However, in the stochastic setting, these vectors cannot be directly computed because the true gradient \(\nabla f(x)\)  is not accessible. The following lemma establishes a relationship between \(\|v_t^{\operatorname{R}}\|\), \(\|v_t^{\P}\|\)and \(\|\zeta_t^{\operatorname{R}}\|\), \(\|\zeta_t^{\P}\|\). This allows \(v_t^{\operatorname{R}}\) and \(v_t^{\P}\) to remain valid theoretical measures of stationarity, even though they cannot be evaluated.
\begin{lemma}\label{lemma.bound of G stochastic}
    Suppose Assumption \ref{assumption} holds. At each iteration $t$, it holds that
    \begin{align*}
        \|v_t^{\operatorname{R}}\|^2 \ &\le \ 2\E \left[\|\zeta_t^{\operatorname{R}}\|^2 \mid x_t \right] + 2\sigma^2 / (\gamma_t^2\lambda^2|\B_t|), \\
        \|v_t^{\P}\|^2 \ &\le \ 2\E \left[\|\zeta_t^{\P}\|^2 \mid x_t \right] + 2\sigma^2 / (\gamma_t^2\lambda^2|\B_t|),
    \end{align*}
    where the expectation is taken with respect to $\pi_{t}^{(1)}, \ldots, \pi_{t}^{(\B_t)}$.
\end{lemma}
\begin{proof}
    To prove the first inequality, we first observe that
  \begin{equation}\label{eq.bound G}
    \|v_t^{\operatorname{R}}\|^2 \ \le \ 2\|\zeta_t^{\operatorname{R}}\|^2 + 2\|v_t^{\operatorname{R}} - \zeta_t^{\operatorname{R}}\|^2.
  \end{equation}
  It therefore suffices to bound the term \(\|v_t^{\operatorname{R}} - \zeta_t^{\operatorname{R}}\|\). Recall $v_t^{\operatorname{R}}, \zeta_t^{\operatorname{R}}$ satisfy
  \begin{align*}
    v_t^{\operatorname{R}} \ &= \ \argmin_{v \in \T_{x_t}\M} \ \inner{\nabla f(x_t)}{v} + \gamma_t D_h(x_t + v, x_t) + g(x_t + v), \\
      \zeta_t^{\operatorname{R}} \ &= \ \argmin_{\zeta \in \T_{x_t}\M} \ \inner{\nabla f_{\B_t}(x_t)}{\zeta} + \gamma_t D_h(x_t + \zeta, x_t) + g(x_t + \zeta).
  \end{align*}
  From the optimality condition of constrained optimization, there exist $s_t \in \partial g(x_t + v_t^{\operatorname{R}})$ and $s_t^\prime \in \partial g(x_t + \zeta_t^{\operatorname{R}})$ such that
  \begin{subequations}
      \begin{align}
                   \inner{\nabla f(x_t) + \gamma_t\nabla h(x_t + v_t^{\operatorname{R}}) - \gamma_t \nabla h(x_t) + s_t}{v - v_t^{\operatorname{R}}} \ \ge \ 0, \ \forall v \in \T_{x_t}\M, \label{subeq.optimal conidtion v}\\ 
      \inner{\nabla f_{\B_t}(x_t) + \gamma_t\nabla h(x_t + \zeta_t^{\operatorname{R}}) - \gamma_t \nabla h(x_t) + s_t^\prime}{\zeta - \zeta_t^{\operatorname{R}}} \ \ge \ 0, \ \forall \zeta \in \T_{x_t}\M. \label{subeq.optimal conidtion zeta}
      \end{align}
  \end{subequations}
    By summing over the above two inequalities with $v = \zeta_t^{\operatorname{R}}$ in equation \eqref{subeq.optimal conidtion v} and $\zeta = v_t^{\operatorname{R}}$ in equation \eqref{subeq.optimal conidtion zeta}, we obtain
    \begin{align*}
        \inner{\nabla f(x_t) - \nabla f_{\B_t}(x_t)}{\zeta_t^{\operatorname{R}} - v_t^{\operatorname{R}}} \ \ge \ \gamma_t\inner{\nabla h(x_t + \zeta_t^{\operatorname{R}}) - \nabla h(x_t + v_t^{\operatorname{R}})}{\zeta_t^{\operatorname{R}} - v_t^{\operatorname{R}}} + \inner{s_t^\prime - s_t}{\zeta_t^{\operatorname{R}} - v_t^{\operatorname{R}}}.
    \end{align*}
   Due to the strong convexity of the reference function $h$, we have $\inner{\nabla h(x_t + \zeta_t^{\operatorname{R}}) - \nabla h(x_t + v_t^{\operatorname{R}})}{\zeta_t^{\operatorname{R}} - v_t} \ge \lambda \|\zeta_t^{\operatorname{R}} - v_t^{\operatorname{R}}\|^2$. Also, $\inner{s_t^\prime - s_t}{\zeta_t^{\operatorname{R}} - v_t^{\operatorname{R}}} \ge 0$ since $g$ is convex. We conclude $\|\nabla f(x_t) - \nabla f_{\B_t}(x_t)\| \cdot \|\zeta_t^{\operatorname{R}} - v_t^{\operatorname{R}}\| \ge \gamma_t \lambda \|\zeta_t^{\operatorname{R}} - v_t^{\operatorname{R}}\|^2$. Substituting the above inequality into equation \eqref{eq.bound G} yields
    \begin{align*}
        \|v_t^{\operatorname{R}}\|^2 \ \le \ 2\|\zeta_t^{\operatorname{R}}\|^2 + 2\frac{\|\nabla f(x_t) - \nabla f_{\B_t}(x_t)\|^2}{\gamma_t^2\lambda^2}.
    \end{align*}
    By the batching property, it follows $\E \left[\|\nabla f(x_t) - \nabla f_{\B_t}(x_t)\|^2 \mid x_t\right] \le \sigma^2 /|\B_t|$, where the expectation is taken over the mini-batch samples $\pi_{t}^{(1)}, \ldots, \pi_{t}^{(\B_t)}$. As a consequence, we obtain
    \begin{align*}
         \|v_t^{\operatorname{R}}\|^2 \ \le \ 2\E \left[\|\zeta_t^{\operatorname{R}}\|^2 \mid x_t \right] + \frac{2\sigma^2}{\gamma_t^2\lambda^2|\B_t|}.
    \end{align*}
    As for the second one, we can use a similar argument. By combining the optimality conditions of $\|v_t^\P\|$ and $\|\xi_t^\P\|$, we have $\|\grad f(x_t) - \grad f_{\B_t}(x_t)\| = \gamma_t \|\nabla h(x_t + v_t^\P) - \nabla h(x_t + \zeta_t^\P)\|$. Again, using the strong convexity of $h$ yields $\gamma_t \lambda \|v_t^\P - \zeta_t^\P\| \le \|\grad f(x_t) - \grad f_{\B_t}(x_t)\|$. Since $\grad f_{\B_t}(x_t) = \P_{\T_{x_t}\M}(\nabla f_{\B_t}(x_t))$, the batching property still holds, and we can conclude
    \begin{align*}
         \|v_t^{\P}\|^2 \ &\le \ 2\E \left[\|\zeta_t^{\P}\|^2 \mid x_t \right] + \frac{2\sigma^2}{\gamma_t^2\lambda^2|\B_t|}.
    \end{align*}
    The proof is completed.
\end{proof}

We now turn to establishing the sample complexity of the two stochastic methods. In the context of stochastic optimization, we assume that $\M$ is a compact Riemannian embedded submanifold. This assumption ensures that gradient-related quantities remain uniformly bounded throughout the analysis. Besides, we can also obtain uniform constants $M_1^{\operatorname{R}}, M_2^{\operatorname{R}} > 0$ in retraction inequalities \eqref{eq.retraction inequality} by \cite{boumal2019global}.
\begin{assumption}\label{assum.bounded iterate}
    The Riemannian embedded submanifold in problem \eqref{eq.main stochastic} is compact. Accordingly, we define $G_f \triangleq \max_{x \in \M}\|\nabla f(x)\|$, and $G_h \triangleq \max_{x \in \M}\|\nabla h(x)\|$.
\end{assumption}
\begin{lemma}\label{lemma.stochastic descent property}
    Suppose Assumptions \ref{assumption} and \ref{assum.bounded iterate} hold. For any stepsize $\gamma_t \ge L$ and $\alpha_t > 0$, the iterate $x_{t+1}$ generated by Algorithm \ref{alg.stochastic relative GD} satisfies
    \begin{align*}
       \E \left[F(x_{t+1}) - F(x_t) \mid x_t \right] \ \le \ -\left(\frac{\gamma_t \lambda\alpha_t}{4} -\left(G_f + 2\gamma_t G_h + L_g\right)M_2^{\operatorname{R}} \alpha_t^2\right)\|\zeta_t^{\operatorname{R}}\|^2 + \frac{\sigma^2 \alpha_t}{\gamma_t \lambda |\B_t|},
   \end{align*}
   where the expectation is taken with respect to $\pi_{t}^{(1)}, \ldots, \pi_{t}^{(\B_t)}$.
\end{lemma}
\begin{proof}
    Following a similar argument in the proof of Lemma \ref{lemma.descent property} implies
   \begin{align*}
       f(x_{t+1}) - f(x_t) \ \le \ \alpha_t\left(\inner{\nabla f(x_t)}{\zeta_t^{\operatorname{R}}} + \gamma_t D_h(x_t + \zeta_t^{\operatorname{R}}, x_t)\right) + \left(G_f + 2\gamma_t G_h\right)M_2^{\operatorname{R}} \|\alpha_t \zeta_t^{\operatorname{R}}\|^2.
   \end{align*}
   It remains to upper the term 
   \begin{align*}
      &\alpha_t \left(\inner{\nabla f(x_t)}{\zeta_t^{\operatorname{R}}} + \gamma_t D_h(x_t + \zeta_t^{\operatorname{R}}, x_t) \right) \\
        = \ &\alpha_t \left(\inner{\nabla f_{\B_t}(x_t)}{\zeta_t^{\operatorname{R}}} + \gamma_t D_h(x_t + \zeta_t^{\operatorname{R}}, x_t) \right) + \alpha_t \inner{\nabla f(x_t) - \nabla f_{\B_t}(x_t)}{\zeta_t^{\operatorname{R}}}. 
   \end{align*}
   From the update of \eqref{eq.stochastic retraction update}, we have
   \begin{align*}
      \inner{\nabla f_{\B_t}(x_t) + \gamma_t \nabla h(x_t + \zeta_t^{\operatorname{R}}) - \gamma_t \nabla h(x_t) + s_t^\prime}{v - \zeta_t^{\operatorname{R}}} \ge 0, \ \forall \zeta \in \T_{x_t}\M,  
   \end{align*}
   where $s_t^\prime \in \partial g(x_t + \zeta_t^{\operatorname{R}})$. Specifically, choose \( \zeta \) to be the zero vector in \( \T_x \mathcal{M} \) and it yields $\inner{\nabla f_{\B_t}(x_t) - \gamma_t \nabla h(x_t)}{\zeta_t^{\operatorname{R}}} \le \inner{\gamma_t \nabla h(x_t + \zeta_t^{\operatorname{R}})}{-\zeta_t^{\operatorname{R}}} - \inner{s_t^\prime}{\zeta_t^{\operatorname{R}}}$. By using the strong convexity of $h$. we have
   \begin{align*}
      \alpha_t \left(\inner{\nabla f_{\B_t}(x_t)}{\zeta_t^{\operatorname{R}}} + \gamma_t D_h(x_t + \zeta_t^{\operatorname{R}}, x_t) \right) \ \le \ -\frac{\alpha_t \gamma_t \lambda}{2}\|\zeta_t^{\operatorname{R}}\|^2 - \alpha_t\inner{s_t^\prime}{\zeta_t^{\operatorname{R}}}.
   \end{align*}
   Besides, we can apply Young's inequality to obtain
   \begin{align*}
      \alpha_t \inner{\nabla f(x_t) - \nabla f_{\B_t}(x_t)}{\zeta_t^{\operatorname{R}}} \ \le \ \frac{\alpha_t}{\gamma_t \lambda}\|\nabla f(x_t) - \nabla f_{\B_t}(x_t)\|^2 + \frac{\gamma_t \lambda \alpha_t}{4}\|\zeta_t^{\operatorname{R}}\|^2.
   \end{align*}
   For the nonsmooth part $g$, it holds that
   \begin{align*}
       g(x_{t+1}) - g(x_t) \ \le \ L_g M_2^{\operatorname{R}}\|\alpha_t \zeta_t^{\operatorname{R}}\|^2 + \alpha_t \inner{s_t^\prime}{\zeta_t^{\operatorname{R}}}.
   \end{align*}
   Therefore, the descent property of $F$ can be shown as
   \begin{align*}
       F(x_{t+1}) - F(x_t) \ \le \ -\left(\frac{\gamma_t \lambda\alpha_t}{4} -\left(G_f + 2\gamma_t G_h + L_g\right)M_2^{\operatorname{R}} \alpha_t^2\right)\|\zeta_t^{\operatorname{R}}\|^2 + \frac{\alpha_t}{\gamma_t \lambda}\|\nabla f(x_t) - \nabla f_{\B_t}(x_t)\|^2.
   \end{align*}
   Finally, it remains to take the expectation
    \begin{align*}
       \E \left[F(x_{t+1}) - F(x_t) \mid x_t \right] \ \le \ -\left(\frac{\gamma_t \lambda\alpha_t}{4} -\left(G_f + 2\gamma_t G_h + L_g\right)M_2^{\operatorname{R}} \alpha_t^2\right)\|\zeta_t^{\operatorname{R}}\|^2 + \frac{\sigma^2 \alpha_t}{\gamma_t \lambda |\B_t|},
   \end{align*}
   and the proof is completed.
\end{proof}

\begin{theorem}\label{thm.sample complexity relative}
    Suppose Assumptions \ref{assumption} and \ref{assum.bounded iterate} hold, and let $T \ge 1$ be the total number of iterations. Under the following parameter setting:
    \begin{equation}\label{eq.relative parameter choice}
       \gamma_t = \gamma \ge L, \ \alpha_t = \alpha < \frac{\gamma \lambda}{8\left(G_f + 2\gamma G_h + L_g\right)M_2^{\operatorname{R}}}, \ |\B_t| = |\B|,  
    \end{equation}
   the sequence $\{x_t\}_{t = 0}^{T-1}$ generated by Algorithm \ref{alg.stochastic relative GD} satisfies 
   \begin{align*}
      \frac{1}{T}\sum_{t=0}^{T-1} \E \left[\|v_t^{\operatorname{R}}\|^2 \right] \ \le \ \mathcal{O}\left(\frac{F(x_0) - F^*}{T} + \frac{\sigma^2}{|\B|}\right),
   \end{align*}
   where the expectation is taken with respect to all the randomness. 
\end{theorem}
\begin{proof}
    From Lemma \ref{lemma.stochastic descent property}, rearranging terms gives
    \begin{align*}
        \left(\frac{\gamma_t \lambda\alpha_t}{4} -\left(G_f + 2\gamma_t G_h + L_g\right)M_2^{\operatorname{R}} \alpha_t^2\right)\|\zeta_t^{\operatorname{R}}\|^2 \ \le \ \E \left[F(x_{t}) - F(x_{t+1}) \mid x_t \right] + \frac{\sigma^2 \alpha_t^2}{2|\B_t|}
    \end{align*}
    for any $\gamma_t \ge L$ and $\alpha_t > 0$. Notice that the choice $\alpha_t \le \gamma_t \lambda / (8\left(G_f + 2\gamma_t G_h + L_g\right)M_2^{\operatorname{R}})$ guarantees that
    \begin{align*}
       \frac{\gamma_t \lambda\alpha_t}{4} -\left(G_f + 2\gamma_t G_h + L_g\right)M_2^{\operatorname{R}} \alpha_t^2 \ \ge \ \frac{\gamma_t \lambda\alpha_t}{8}.
    \end{align*}
    Consequently, we obtain
    \begin{align*}
        \|\zeta_t^{\operatorname{R}}\|^2 \le \frac{8\E \left[F(x_{t}) - F(x_{t+1}) \mid x_t \right]}{\gamma_t \lambda\alpha_t} + \frac{4\sigma^2 \alpha_t}{\gamma_t \lambda|\B_t|}.
    \end{align*}
    By combining with Lemma \ref{lemma.bound of G stochastic}, and substituting the parameter choice \eqref{eq.relative parameter choice}, it follows
    \begin{align*}
        \alpha_t \|v_t^{\operatorname{R}}\|^2 \ \le \ \frac{16\E \left[F(x_{t}) - F(x_{t+1}) \mid x_t \right]}{\gamma_t \lambda\alpha_t} + \frac{8\sigma^2 \alpha_t}{\gamma_t \lambda|\B_t|} + \frac{2\sigma^2}{\gamma_t^2\lambda^2|\B_t|}.
    \end{align*}
    summing over $t$ from $0$ to $T-1$ imply
    \begin{align*}
        \frac{1}{T}\sum_{t=0}^{T-1} \E \left[\|v_t^{\operatorname{R}}\|^2 \right] \  \le \ \frac{16(F(x_0) - F(x^*))}{\gamma \lambda\alpha T} + \frac{8\sigma^2 \alpha}{\gamma \lambda|\B|} + \frac{2\sigma^2}{\gamma^2\lambda^2|\B|}
    \end{align*}
    where the expectation is taken with respect to all the randomness. The proof is completed.
\end{proof}

Given an accuracy $\epsilon > 0$, choose a minibatch size $|\B| = \mathcal{O}(\epsilon^{-2})$. Then, after $T = \mathcal{O}(\epsilon^{-2})$ iterations, an $\epsilon$-approximate Riemannian stationary point can be found in expectation, and the overall sample complexity is $\mathcal{O}(\epsilon^{-4})$. Now we establish the sample complexity bound of the stochastic projection-based method (Algorithm \ref{alg.stochastic project GD}). As discussed in Remark~\ref{remark.compare with ding}, when $\M$ is compact, we can use the stronger Lemma \ref{lemma.ding} in place of Lemma \ref{lemma.projection inequality}, and hence no bound on the tangent component is required. Let $H_h \triangleq \max_{x \in \M} \|\nabla^2 h(x)\|$. Define $\Psi_1 \triangleq 2G_h(M_{2}^\P + M_{3}^\P) + H_h$ and $\Psi_2 \triangleq \left((M_{2}^\P + M_{3}^\P) + (M_{1}^\P)^2M_4^\P\right)G_f$, where $M_i^\P$ ($i = 1,2,3$) are the constants in Lemma \ref{lemma.ding}. Therefore, since $u_t = 0$ in Algorithm \ref{alg.stochastic project GD}, we can conclude the following descent property
\begin{equation}\label{eq.descent stochastic projection}
    f(x_{t+1}) - f(x_t) \ \le \
    \left(\gamma_t\Psi_1 + \Psi_2\right)\|\alpha_t \zeta_t^{\P}\|^2 + \alpha_t \inner{\grad f(x_t)}{\zeta_t^{\P}} + \alpha_t \gamma_t D_h(x_t + \zeta_t^{\P}, x_t) 
\end{equation}
by using a similar argument in the proof of Theorem~\ref{thm.convergence of projection based}.

\begin{theorem}\label{thm.sample complexity projection}
    Suppose Assumptions \ref{assumption}, \ref{assumption.correction u} and \ref{assum.bounded iterate} hold. Let $T \ge 1$ be the total number of iterations. Under the following parameter setting:
    \begin{equation}\label{eq.projection parameter choice}
       \gamma_t = \gamma \ge L, \ \alpha_t = \alpha \le  \frac{\gamma \lambda}{8(\gamma \Psi_1 + \Psi_2)}, \ |\B_t| = |\B|,  
    \end{equation}
  the sequence $\{x_t\}_{t \ge 0}$ generated by Algorithm \ref{alg.stochastic project GD} satisfies 
    \begin{align*}
        \frac{1}{T}\sum_{t=0}^{T-1} \E \left[\|v_t^{\operatorname{\P}}\|^2 \right] \ \le \ \mathcal{O}\left(\frac{F(x_0) - F^*}{T} + \frac{\sigma^2}{|\B|}\right).
    \end{align*}
  where the expectation is taken with respect to all the randomness. 
\end{theorem}
\begin{proof}
    Recall that $g \equiv 0$. From \eqref{eq.descent stochastic projection}, it holds that
   \begin{align*}
       F(x_{t+1}) - F(x_t) \ \le \
    \left(\gamma_t\Psi_1 + \Psi_2\right)\|\alpha_t \zeta_t^{\P}\|^2 + \alpha_t \inner{\grad f(x_t)}{\zeta_t^{\P}} + \alpha_t \gamma_t D_h(x_t + \zeta_t^{\P}, x_t).
   \end{align*} 
   Then, we decompose $\grad f(x_t) = \grad f_{\B_t}(x_t) + \grad f(x_t) - \grad f_{\B_t}(x_t)$. By using the optimality condition of the subproblem \ref{eq.stochastic projection update}, we obtain $\grad f_{\B_t}(x_t) = \gamma_t \nabla h(x_t) - \gamma_t \nabla h(x_t + \zeta_t^{\P})$, which further implies
   \begin{align*}
       &\alpha_t \inner{\grad f(x_t)}{\zeta_t^{\P}} + \alpha_t \gamma_t D_h(x_t + \zeta_t^{\P}, x_t) \\
       = \ &\alpha_t \inner{\gamma_t \nabla h(x_t) - \gamma_t \nabla h(x_t + \zeta_t^{\P})}{\zeta_t^{\P}} + \alpha_t \gamma_t D_h(x_t + \zeta_t^{\P}, x_t) + \alpha_t \inner{\grad f(x_t) - \grad f_{\B_t}(x_t)}{\zeta_t^{\P}} \\
       = \ &\alpha_t \gamma_t \left(h(x_t + \zeta_t^{\P}) - h(x_t) - \inner{\nabla h(x_t + \zeta_t^{\P})}{\zeta_t^{\P}}\right) + \alpha_t \inner{\grad f(x_t) - \grad f_{\B_t}(x_t)}{\zeta_t^{\P}} \\
       \le \ &-\frac{\alpha_t \gamma_t \lambda}{2}\|\zeta_t^\P\|^2 + \frac{\alpha_t}{\gamma_t \lambda}\|\grad f(x_t) - \grad f_{\B_t}(x_t)\|^2 + \frac{\gamma_t \lambda \alpha_t}{4}\|\zeta_t^{\P}\|^2 \\
       = \ &-\frac{\alpha_t \gamma_t \lambda}{4}\|\zeta_t^\P\|^2 + \frac{\alpha_t}{\gamma_t \lambda}\|\grad f(x_t) - \grad f_{\B_t}(x_t)\|^2.
   \end{align*}
   As a result, the descent property follows
   \begin{align*}
      \E \left[F(x_{t+1}) - F(x_t) \mid x_t \right] \ \le \ -\left(\frac{\gamma_t \lambda\alpha_t}{4} -\left(\gamma_t\Psi_1 + \Psi_2\right)\alpha_t^2\right)\|\zeta_t^{\P}\|^2 + \frac{\sigma^2 \alpha_t^2}{2|\B_t|}. 
   \end{align*}
   Notice that under the choice $\alpha_t \le \gamma \lambda / (8(\gamma \Psi_1 + \Psi_2))$, we have
    \begin{align*}
       \frac{\gamma_t \lambda\alpha_t}{4} -(\gamma \Psi_1 + \Psi_2) \alpha_t^2 \ \ge \ \frac{\gamma_t \lambda\alpha_t}{8}.
    \end{align*}
    Consequently, we obtain
    \begin{align*}
        \|\zeta_t^{\operatorname{\P}}\|^2 \le \frac{8\E \left[F(x_{t}) - F(x_{t+1}) \mid x_t \right]}{\gamma_t \lambda\alpha_t} + \frac{4\sigma^2 \alpha_t}{\gamma_t \lambda|\B_t|}.
    \end{align*}
    By combining with Lemma \ref{lemma.bound of G stochastic} and substituting the parameter choice \eqref{eq.relative parameter choice}, we can conclude
    \begin{align*}
          \frac{1}{T}\sum_{t=0}^{T-1} \E \left[\|v_t^{\operatorname{\P}}\|^2 \right] \ \le \ \mathcal{O}\left(\frac{F(x_0) - F^*}{T} + \frac{\sigma^2}{|\B|}\right).
       \end{align*}
       The proof is completed.
\end{proof}

\section{Numerical Experiments}\label{section.numerics}
In this section, we numerically test our Riemannian Bregman gradient methods on the nonlinear eigenvalue problem \eqref{eq.ne} and the low-rank quadratic sensing problem \eqref{eq.low-rank recover}. All numerical experiments reported here are performed on a platform equipped with an Apple M1 CPU and 8GB of RAM. We test Algorithm~\ref{alg.relative GD} (R-RBGD), Algorithm~\ref{alg.projection relative GD} without a corrective normal vector (P-RBGD), and Algorithm~\ref{alg.projection relative GD} with a corrective normal vector $u_t = -v_t^\N$ (P-BRGD-C). We compare our methods against the \texttt{steepestdescent} solver in Manopt~\citep{boumal2014manopt}, employing both the default linesearch (RSD) and adaptive linesearch (RSD-Ada) strategies. For RSD and RSD-Ada, all parameters are kept at their default values provided by the package. The initial points are randomly generated.

We begin by testing the low-rank quadratic sensing problem \eqref{eq.low-rank recover}. The data $\{(y_j,c_j)\}_{j=1}^N$ are randomly generated by MATLAB's \texttt{randn} function, and we set $N = 100$. By proposition \ref{prop.projection fixed rank}, the update directions generated by R-RBGD and P-RBGD coincide. Hence, we only test P-RBGD, as the associated subproblem is computationally simpler. The parameters in the backtracking linesearch are set to an initial stepsize of 0.1 and a contraction factor of 0.5. All algorithms are terminated when the norm of the Riemannian gradient is less than $10^{-4}$. We test all algorithms with various parameter combinations of $m$ and $r$. The results are illustrated in Figures \ref{fig:problem12_m500_ranks}, \ref{fig:problem12_m1000_ranks}, \ref{fig:problem12_m2000_ranks}, \ref{fig:problem12_m4000_ranks}. From these figures, we observe that our projection-based Riemannian Bregman gradient method outperforms RSD and RSD-Ada in most cases, requiring fewer iterations to achieve the specified accuracy.
\begin{figure}[H]
    \begin{center}
        \minipage{0.33\textwidth}
            \subfigure[$m = 500, r = 10$]{%
                \includegraphics[width=0.9\linewidth]{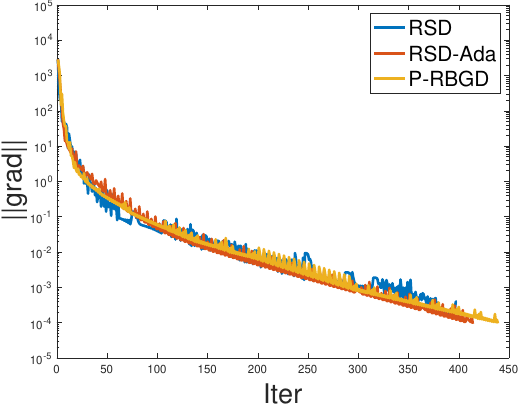}%
            }
        \endminipage\hfill
        \minipage{0.33\textwidth}
            \subfigure[$m = 500, r = 20$]{%
                \includegraphics[width=0.9\linewidth]{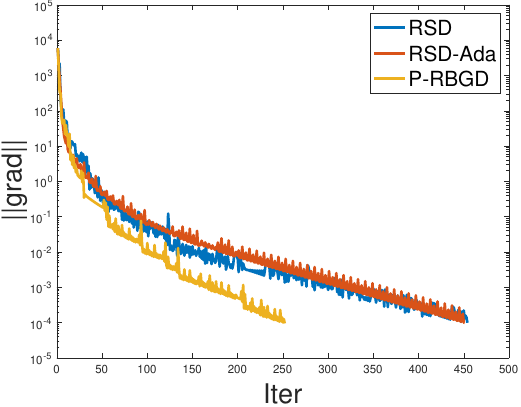}%
            }
        \endminipage\hfill
        \minipage{0.33\textwidth}
            \subfigure[$m = 500, r = 40$]{%
                \includegraphics[width=0.9\linewidth]{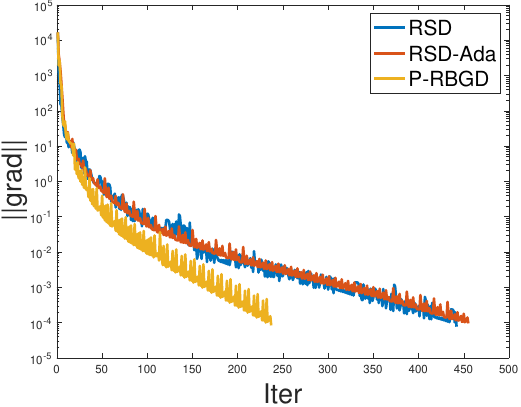}%
            }
        \endminipage
        \caption{The results of problem~\eqref{eq.low-rank recover} with $m=500$ and varying rank $r = 10, 20, 40$. }
        \label{fig:problem12_m500_ranks}
    \end{center}
\end{figure}

\begin{figure}[H]
    \begin{center}
        \minipage{0.33\textwidth}
            \subfigure[$m = 1000, r = 10$]{%
                \includegraphics[width=0.9\linewidth]{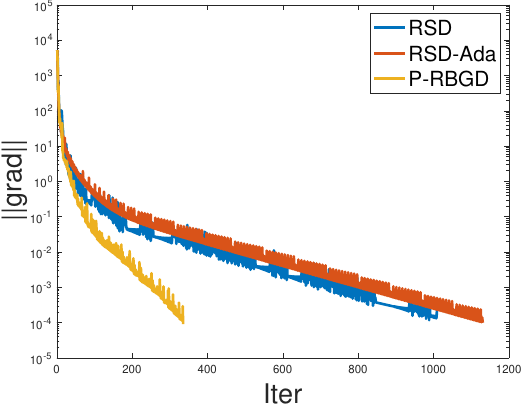}%
            }
        \endminipage\hfill
        \minipage{0.33\textwidth}
            \subfigure[$m = 1000, r = 20$]{%
                \includegraphics[width=0.9\linewidth]{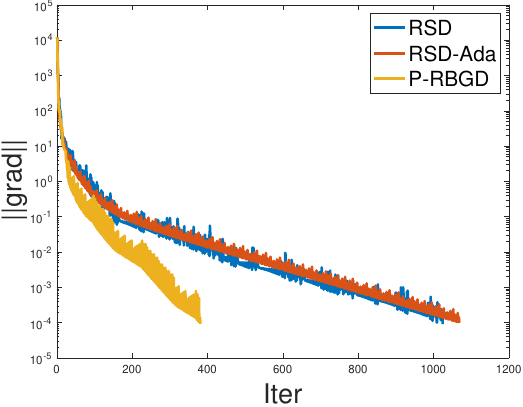}%
            }
        \endminipage\hfill
        \minipage{0.33\textwidth}
            \subfigure[$m = 1000, r = 40$]{%
                \includegraphics[width=0.9\linewidth]{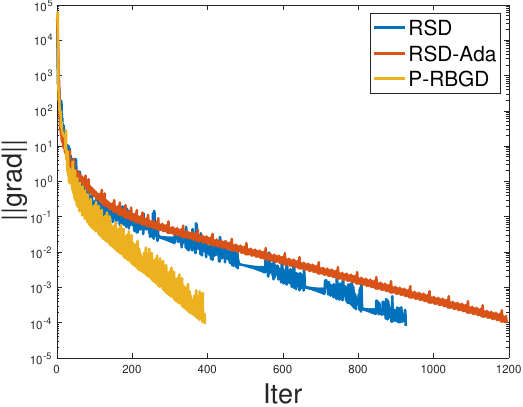}%
            }
        \endminipage
        \caption{The results of problem~\eqref{eq.low-rank recover} with $m=1000$ and varying rank $r = 10, 20, 40$.}
        \label{fig:problem12_m1000_ranks}
    \end{center}
\end{figure}

\begin{figure}[H]
    \begin{center}
        \minipage{0.33\textwidth}
            \subfigure[$m = 2000, r = 10$]{%
                \includegraphics[width=0.9\linewidth]{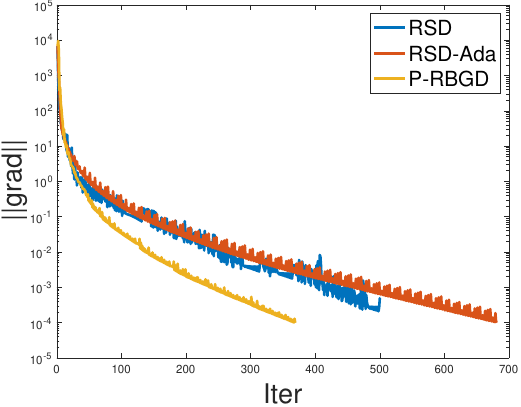}%
            }
        \endminipage\hfill
        \minipage{0.33\textwidth}
            \subfigure[$m = 2000, r = 20$]{%
                \includegraphics[width=0.9\linewidth]{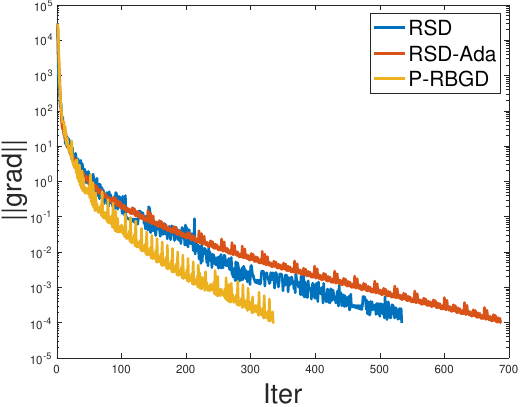}%
            }
        \endminipage\hfill
        \minipage{0.33\textwidth}
            \subfigure[$m = 2000, r = 40$]{%
                \includegraphics[width=0.9\linewidth]{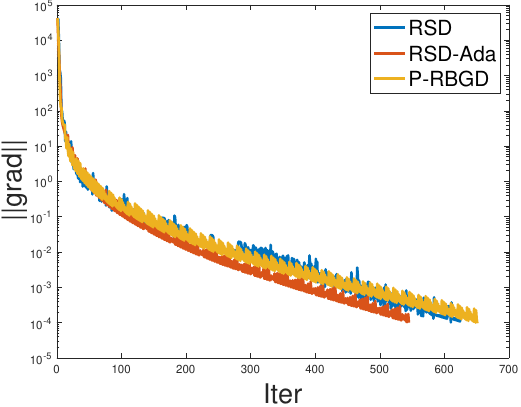}%
            }
        \endminipage
        \caption{The results of problem~\eqref{eq.low-rank recover} with $m=2000$ and varying rank $r = 10, 20, 40$.}
        \label{fig:problem12_m2000_ranks}
    \end{center}
\end{figure}

\begin{figure}[H]
    \begin{center}
        \minipage{0.33\textwidth}
            \subfigure[$m = 4000, r = 10$]{%
                \includegraphics[width=0.9\linewidth]{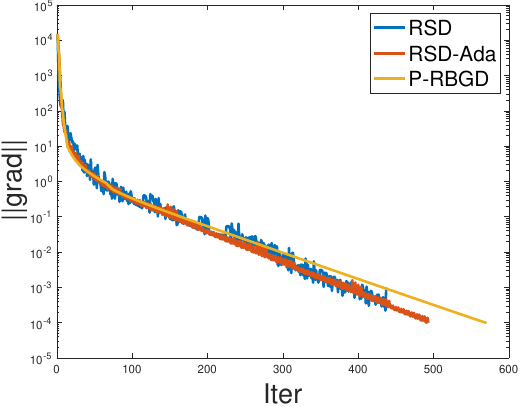}%
            }
        \endminipage\hfill
        \minipage{0.33\textwidth}
            \subfigure[$m = 4000, r = 20$]{%
                \includegraphics[width=0.9\linewidth]{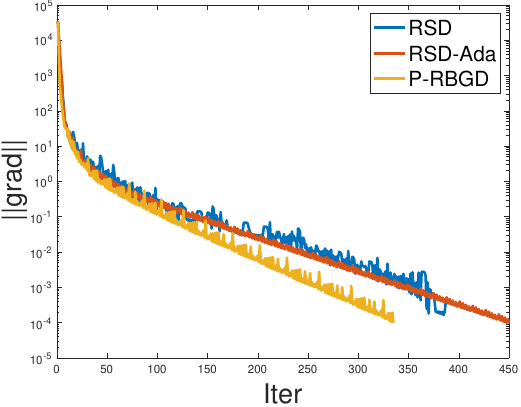}%
            }
        \endminipage\hfill
        \minipage{0.33\textwidth}
            \subfigure[$m = 4000, r = 40$]{%
                \includegraphics[width=0.9\linewidth]{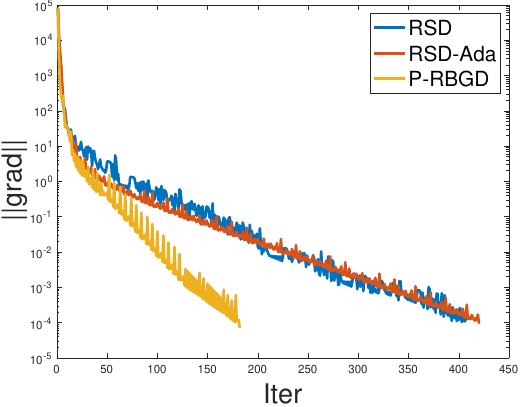}%
            }
        \endminipage
        \caption{The results of problem~\eqref{eq.low-rank recover} with $m=4000$ and varying rank $r = 10, 20, 40$.}
        \label{fig:problem12_m4000_ranks}
    \end{center}
\end{figure}

Next, we apply our methods to the nonlinear eigenvalue problem \eqref{eq.ne}. The parameter $\beta$ is set to $10$. For our three algorithms, we use a backtracking linesearch with an initial stepsize of 0.5 and a contraction factor of 0.5.  Similar to the previous experiment, all algorithms are terminated when the norm of the Riemannian gradient is smaller than $10^{-4}$. Tables~\ref{table.ne m} and~\ref{table.ne p} report the performance of all compared solvers on problem~\eqref{eq.ne} for various combinations of the parameters $m$ and $p$. From these tables, we observe that our Riemannian Bregman gradient methods achieve function values comparable to those obtained by RSD and RSD-Ada, yet require fewer iterations (the function values differ only in the 8th decimal place). Due to the need to solve an auxiliary subproblem at each iteration, the CPU time is similar to that of RSD and RSD-Ada. However, for challenging instances (e.g., $m = 5000, p = 60$), RSD and RSD-Ada fail to satisfy the termination criterion because the stepsize becomes excessively small, causing premature termination.

\begin{table}[htbp]
    \caption{The results of problem \eqref{eq.ne} with varying $m$.}
    \label{table.ne m}
    \centering
    \small
    \newcommand{\timebox}[1]{\makebox[3em][c]{#1}}  
    \newcommand{\iterbox}[1]{\makebox[2em][c]{#1}} 
\begin{tabular}{l|cccc|cccc}
    \toprule
    \multirow{2}{*}{Solver} & \multicolumn{4}{c|}{$m=500,\ p=50$} & \multicolumn{4}{c}{$m=1000,\ p=50$} \\
    \cmidrule(lr){2-5} \cmidrule(lr){6-9}
    & Fval & $\|\grad\|$ & Iter & Time & Fval & $\|\grad\|$ & Iter & Time \\
    \midrule
    RSD      & 2.7674e+04 & 5.1152e-05 & \iterbox{7566} & \timebox{36.2087} & 2.7674e+04 & 2.5639e-04 & \iterbox{8873} & \timebox{68.4400} \\
    RSD-Ada  & 2.7674e+04 & 4.3667e-04 & \iterbox{7650} & \timebox{40.2719} & 2.7674e+04 & 3.8527e-04 & \iterbox{8988} & \timebox{70.3007} \\
    R-RBGD   & 2.7674e+04 & 9.6010e-05 & \iterbox{4938} & \timebox{23.5305} & 2.7674e+04 & 9.8616e-05 & \iterbox{5518} & \timebox{53.8013} \\
    P-RBGD   & 2.7674e+04 & 9.9854e-05 & \iterbox{4864} & \timebox{40.7854} & 2.7674e+04 & 9.9820e-05 & \iterbox{5503} & \timebox{74.4051} \\
    P-RBGD-C & 2.7674e+04 & 9.7726e-05 & \iterbox{4901} & \timebox{46.7791} & 2.7674e+04 & 9.8943e-05 & \iterbox{5510} & \timebox{80.5897} \\
    \bottomrule
\end{tabular}

\begin{tabular}{l|cccc|cccc}
    \toprule
    \multirow{2}{*}{Solver} & \multicolumn{4}{c|}{$m=1500,\ p=50$} & \multicolumn{4}{c}{$m=2000,\ p=50$} \\
    \cmidrule(lr){2-5} \cmidrule(lr){6-9}
    & Fval & $\|\grad\|$ & Iter & Time & Fval & $\|\grad\|$ & Iter & Time \\
    \midrule
    RSD      & 2.7674e+04 & 9.1536e-04 & \iterbox{7108} & \timebox{117.2364} & 2.7674e+04 & 2.6418e-04 & \iterbox{8902} & \timebox{159.2698} \\
    RSD-Ada  & 2.7674e+04 & 4.7544e-04 & \iterbox{6304} & \timebox{113.7995} & 2.7674e+04 & 4.3806e-04 & \iterbox{9115} & \timebox{196.9259} \\
    R-RBGD   & 2.7674e+04 & 9.9155e-05 & \iterbox{5160} & \timebox{102.2899} & 2.7674e+04 & 9.5463e-05 & \iterbox{5964} & \timebox{192.0930} \\
    P-RBGD   & 2.7674e+04 & 9.9829e-05 & \iterbox{5144} & \timebox{130.3309} & 2.7674e+04 & 9.8332e-05 & \iterbox{5884} & \timebox{171.8766} \\
    P-RBGD-C & 2.7674e+04 & 9.8907e-05 & \iterbox{5322} & \timebox{133.2718} & 2.7674e+04 & 9.7909e-05 & \iterbox{6019} & \timebox{177.7956} \\
    \bottomrule
\end{tabular}

\begin{tabular}{l|cccc|cccc}
    \toprule
    \multirow{2}{*}{Solver} & \multicolumn{4}{c|}{$m=2500,\ p=50$} & \multicolumn{4}{c}{$m=3000,\ p=50$} \\
    \cmidrule(lr){2-5} \cmidrule(lr){6-9}
    & Fval & $\|\grad\|$ & Iter & Time & Fval & $\|\grad\|$ & Iter & Time \\
    \midrule
    RSD      & 2.7674e+04 & 2.3527e-04 & \iterbox{7840} & \timebox{206.9153} & 2.7674e+04 & 2.1955e-04 & \iterbox{7732} & \timebox{221.2304} \\
    RSD-Ada  & 2.7674e+04 & 3.5709e-04 & \iterbox{7238} & \timebox{232.8811} & 2.7674e+04 & 5.1092e-04 & \iterbox{6015} & \timebox{204.8033} \\
    R-RBGD   & 2.7674e+04 & 9.9000e-05 & \iterbox{5197} & \timebox{204.8740} & 2.7674e+04 & 9.8652e-05 & \iterbox{5315} & \timebox{239.3315} \\
    P-RBGD   & 2.7674e+04 & 9.8332e-05 & \iterbox{5246} & \timebox{179.6978} & 2.7674e+04 & 9.8969e-05 & \iterbox{5265} & \timebox{213.2356} \\
    P-RBGD-C & 2.7674e+04 & 9.7069e-05 & \iterbox{5303} & \timebox{207.1347} & 2.7674e+04 & 9.8572e-05 & \iterbox{5233} & \timebox{226.3744} \\
    \bottomrule
\end{tabular}
\end{table}

\begin{table}[htbp]
    \caption{The results of problem \eqref{eq.ne} with varying $p$.}
    \label{table.ne p}
    \centering
    \small
    \newcommand{\timebox}[1]{\makebox[3em][c]{#1}}  
    \newcommand{\iterbox}[1]{\makebox[2em][c]{#1}} 
    \begin{tabular}{l|cccc|cccc}
    \toprule
    \multirow{2}{*}{Solver} & \multicolumn{4}{c|}{$m=5000,\ p=10$} & \multicolumn{4}{c}{$m=5000,\ p=20$} \\
    \cmidrule(lr){2-5} \cmidrule(lr){6-9}
    & Fval & $\|\grad\|$ & Iter & Time & Fval & $\|\grad\|$ & Iter & Time \\
    \midrule
    RSD      & 2.8429e+02 & 8.1462e-05 & \iterbox{249}  & \timebox{1.0864} & 1.9443e+03 & 9.9204e-05 & \iterbox{1333} & \timebox{9.7795} \\
    RSD-Ada  & 2.8429e+02 & 9.8615e-05 & \iterbox{307}  & \timebox{1.2623} & 1.9443e+03 & 9.8288e-05 & \iterbox{1401} & \timebox{10.5765} \\
    R-RBGD   & 2.8429e+02 & 9.3002e-05 & \iterbox{317}  & \timebox{1.3235} & 1.9443e+03 & 9.7837e-05 & \iterbox{1282} & \timebox{12.6708} \\
    P-RBGD   & 2.8429e+02 & 9.3140e-05 & \iterbox{315}  & \timebox{1.1404} & 1.9443e+03 & 9.9858e-05 & \iterbox{1274} & \timebox{11.4878} \\
    P-RBGD-C & 2.8429e+02 & 9.5805e-05 & \iterbox{314}  & \timebox{1.3440} & 1.9443e+03 & 9.8893e-05 & \iterbox{1267} & \timebox{12.6460} \\
    \bottomrule
\end{tabular}

    \begin{tabular}{l|cccc|cccc}
    \toprule
    \multirow{2}{*}{Solver} & \multicolumn{4}{c|}{$m=5000,\ p=30$} & \multicolumn{4}{c}{$m=5000,\ p=40$} \\
    \cmidrule(lr){2-5} \cmidrule(lr){6-9}
    & Fval & $\|\grad\|$ & Iter & Time & Fval & $\|\grad\|$ & Iter & Time \\
    \midrule
    RSD      & 6.2293e+03 & 9.9957e-05 & \iterbox{3328} & \timebox{77.0571} & 1.4389e+04 & 9.5968e-05 & \iterbox{5728} & \timebox{171.1722} \\
    RSD-Ada  & 6.2293e+03 & 9.9447e-05 & \iterbox{3228} & \timebox{84.4482} & 1.4389e+04 & 3.1914e-04 & \iterbox{3872} & \timebox{126.8553} \\
    R-RBGD   & 6.2293e+03 & 9.9393e-05 & \iterbox{1681} & \timebox{43.5144} & 1.4389e+04 & 9.8178e-05 & \iterbox{5367} & \timebox{241.8630} \\
    P-RBGD   & 6.2293e+03 & 9.9030e-05 & \iterbox{1679} & \timebox{38.0242} & 1.4389e+04 & 9.6651e-05 & \iterbox{5356} & \timebox{218.3186} \\
    P-RBGD-C & 6.2293e+03 & 9.9961e-05 & \iterbox{1665} & \timebox{40.7001} & 1.4389e+04 & 9.8343e-05 & \iterbox{5339} & \timebox{223.8977} \\
    \bottomrule
\end{tabular}

\begin{tabular}{l|cccc|cccc}
    \toprule
    \multirow{2}{*}{Solver} & \multicolumn{4}{c|}{$m=5000,\ p=50$} & \multicolumn{4}{c}{$m=5000,\ p=60$} \\
    \cmidrule(lr){2-5} \cmidrule(lr){6-9}
    & Fval & $\|\grad\|$ & Iter & Time & Fval & $\|\grad\|$ & Iter & Time \\
    \midrule
    RSD      & 2.7674e+04 & 2.7698e-04 & \iterbox{8007} & \timebox{277.4160} & 4.7334e+04 & 1.3049e-03 & \iterbox{10090} & \timebox{392.3834} \\
    RSD-Ada  & 2.7674e+04 & 4.3422e-04 & \iterbox{7477} & \timebox{300.1134} & 4.7334e+04 & 6.6243e-04 & \iterbox{10260} & \timebox{415.5020} \\
    R-RBGD   & 2.7674e+04 & 9.9805e-05 & \iterbox{5326} & \timebox{328.1556} & 4.7334e+04 & 9.9850e-05 & \iterbox{5128} & \timebox{353.4060} \\
    P-RBGD   & 2.7674e+04 & 9.9563e-05 & \iterbox{5228} & \timebox{303.5396} & 4.7334e+04 & 9.9838e-05 & \iterbox{5267} & \timebox{382.2308} \\
    P-RBGD-C & 2.7674e+04 & 9.7795e-05 & \iterbox{5423} & \timebox{337.2280} & 4.7334e+04 & 9.8905e-05 & \iterbox{5300} & \timebox{411.1444} \\
    \bottomrule
\end{tabular}
\end{table}

\section{Conclusions}
In this paper, we developed two Riemannian Bregman gradient methods for solving relatively smooth optimization problems over Riemannian embedded submanifolds. The retraction-based method handles nonsmooth optimization by solving a convex subproblem constrained to the tangent space at each iteration. We identified particular reference functions, such as the quartic, log-barrier, and entropy functions, for which the subproblem admits either a closed-form solution or significantly simplified one. The projection-based approach, suitable for smooth optimization, involves solving an unconstrained subproblem in the ambient Euclidean space followed by a projection onto the manifold. Both methods achieve an iteration complexity of $\mathcal{O}(1/\epsilon^2)$ for finding an $\epsilon$-approximate Riemannian stationary point. Additionally, for compact manifolds, we proposed stochastic variants with sample complexities of $\mathcal{O}(1/\epsilon^4)$. Numerical experiments demonstrated the effectiveness of proposed Riemannian Bregman gradient methods.

\clearpage
\section*{Appendix}
In the appendix, we first introduce several concepts from variational analysis that are used in the proof of Lemma~\ref{lemma.continuity of v*}. 
Then, we restate the well-known Berge's Maximum Theorem, which is utilized in the proof of Theorem~\ref{thm.convergence of retraction based}, and Lemma~5.10 from \cite{ding2024convergence}, which provides inequalities related to projections onto a compact submanifold.

\begin{definition}[Set-valued map]\label{def.set valued map}
Let $X$ and $Y$ be topological spaces. A set-valued map from $X$ to $Y$ is a mapping $S: X \rightrightarrows Y$ that assigns to each $x \in X$ a subset $S(x) \subseteq Y$.
\end{definition}

\begin{definition}[Continuity of set-valued maps]\label{def.continuity of set valued map}
Let $X$ and $Y$ be metric spaces, and $S: X \rightrightarrows Y$ a set-valued map. We say $S$ is outer semicontinuous at $x \in X$ if, whenever $x_k \to x$ and $y_k \to y$ with $y_k \in S(x_k)$, we have $y \in S(x)$; $S$ is inner semicontinuous at $x \in X$ if, for any sequence $x_k \to x$ and any $y \in S(x)$, there exists a sequence $y_k \in S(x_k)$ with $y_k \to y$. $S$ is continuous at $x$ if it is both outer and inner semicontinuous at $x$.
\end{definition}

\begin{theorem}[Berge’s Maximum Theorem]\label{thm.Berge}
Let \(S \colon \mathbb{R}^n \rightrightarrows \mathbb{R}^n\) be a non-empty, compact-valued, continuous set-valued map, and let \(\varphi \colon \mathbb{R}^n \times \mathbb{R}^n \to \mathbb{R}\) be continuous. Then the maximum value function
\[
   \Phi(x) \triangleq \max_{v \in S(x)} \varphi(x,v), \qquad x \in \mathbb{R}^n
\]
is well-defined and continuous.
\end{theorem}

\begin{lemma}[Restatement of Lemma 5.10 in \cite{ding2024convergence}]\label{lemma.ding}
Let $\mathcal{M} \subseteq \mathbb{R}^n$ be a compact submanifold of class $C^3$. Then there exists a constant $\varrho > 0$ such that for all $x \in \mathcal{M}, v \in \T_{x} \mathcal{M}$ and $u \in \N_{x} \mathcal{M}$ satisfying $\|u\| \leq \varrho / 2$, we have
$$
\begin{aligned}
\left\|\mathcal{P}_{\mathcal{M}}(x+v+u)-x\right\| & \leq M_1^\P\|v\|, \\
\left\|\mathcal{P}_{\mathcal{M}}(x+v+u)-x-v\right\| & \leq M_2^\P\|v\|^2+M_3^\P\|v\|\|u\|,
\end{aligned}
$$
for some positive constants $M_{1}^\P, M_{2}^\P, M_{3}^\P > 0$.
\end{lemma}

\section*{Acknowledgment}
Chang He gratefully acknowledges Xiaotian Jiang at the University of Minnesota for fruitful discussions on projection properties, and Chuwen Zhang at the University of Chicago for helpful suggestions on numerical implementation.

\bibliographystyle{abbrvnat} 
\bibliography{reference}

\end{document}